\documentclass[12pt]{article}
\newcommand{\ds}{\displaystyle}

\usepackage{graphics}
\usepackage{graphicx}
\usepackage{color}

\usepackage{amsfonts}
\usepackage{eufrak}

\usepackage{amsmath}
\usepackage{amstext}
\usepackage{amsopn}
\usepackage{amsbsy}
\usepackage{amscd}
\usepackage{amsxtra}
\usepackage{amsthm}
\usepackage{enumerate}

\numberwithin{equation}{section}

\renewcommand{\abstractname}

\title{Orbital stability of bound states of nonlinear
Schr\"{o}dinger equations with linear and nonlinear optical
lattices}

\author{Tai-Chia Lin  \thanks{ Department of Mathematics, National Taiwan
University, Taipei, Taiwan 106. Email: tclin@math.ntu.edu.tw}
\thanks{Taida Institute of Mathematical Sciences (TIMS), Taipei, Taiwan
}\,,
 Juncheng Wei
\thanks{Department of Mathematics, The Chinese University of Hong Kong,
Shatin, Hong Kong. Email: wei@math.cuhk.edu.hk}\ and Wei Yao
\thanks{Department of Mathematics, The Chinese University of Hong
Kong, Shatin, Hong Kong. Email: wyao@math.cuhk.edu.hk } }
\date{}

\newtheorem{thm}{Theorem}[section]

\newtheorem{lem}[thm]{Lemma}

\begin{document}
\maketitle

\begin{abstract}
We study the orbital stability and instability of single-spike
bound states of semi-classical nonlinear Schr\"{o}dinger (NLS)
equations with critical exponent, linear and nonlinear optical
lattices (OLs). These equations may model two-dimensional
Bose-Einstein condensates in linear and nonlinear OLs. When linear
OLs are switched off, we derive the asymptotic expansion formulas
and obtain necessary conditions for the orbital stability and
instability of single-spike bound states, respectively. When
linear OLs are turned on, we consider three different conditions
of linear and nonlinear OLs to develop mathematical theorems which
are most general on the orbital stability problem.
\end{abstract}

\section{Introduction}
 \ \ \ \ Recently, optical lattices have created many interesting
phenomena in Bose-Einstein condensates (BECs) and attracted a
great deal of attention. Two types of optical lattices are
considered: a linear optical lattice (OL) (cf.~\cite{[1]}) and a
nonlinear OL (cf.~\cite{[8]} and \cite{[7]}). A linear OL is a
series of potential wells having a periodic (in space) intensity
pattern which may confine atoms of BECs in the potential minima. A
nonlinear OL can be obtained by inducing a periodic spatial
variation of the atomic scattering length, leading to a periodic
space modulation of the nonlinear coefficient in the
Gross-Pitaevskii equation (GPE) governing the dynamics of BECs.
The GPE is a nonlinear Schr\"{o}dinger (NLS) equation in the
presence of the Kerr nonlinearity describing a BEC in a linear and
a nonlinear OL given by
\begin{equation}\label{gp1} -i\frac{\partial\psi}{\partial
t}=D\Delta\psi-V_{trap}\psi-g|\psi|^2\psi\,,
\end{equation} for $x\in\mathbb{R}^N$, $N\leq 3$ and $t>0$. Here
$\psi=\psi(x,t)\in\mathbb{C}$ is the wavefunction, $D$ is the
diffraction (or dispersion) coefficient, and $V_{trap}$ is the
potential of the linear lattice. Besides, $g=\mu m(x)\sim a$
characterizes the nonlinear lattice, where $a$ denotes the
spatially modulated scattering length, $\mu$ is a nonzero constant
and $m(x)=m(x_1,\cdots,x_N)>0$ is a function depending on spatial
variables (transverse coordinates)
$x_1,\cdots,x_N$~(cf.~\cite{ASK}, \cite{CLS}).

The underlying dynamics of (\ref{gp1}) is dominated by the
interplay between adjacent potential wells of linear OLs and
nonlinearity of nonlinear OLs. When the nonlinearity is
self-focusing i.e.~$D>0$ and $\mu<0$, a balance between these two
effects may resist collapse or decay and result in bright
solitons. Experimentally, bright solitons can be observed in
linear and nonlinear OLs, respectively. One may find stable bright
solitons in three-dimensional linear OLs (cf.~\cite{CTW}). On the
other hand, two-dimensional bright solitons can also be
investigated in two-dimensional nonlinear OLs (cf.~\cite{FSE}).
Consequently, under the influence of linear and nonlinear OLs,
two-dimensional bright solitons must have suitable stability for
experimental observations. However, most theoretical results
(e.g.~\cite{FSW} and \cite{FSW2}) focus on the orbital (dynamical)
stability of only one-dimensional single-spike bound states which
are steady state bright solitons in one-dimensional nonlinear OLs
without the effect of linear OLs. To see how linear and nonlinear
OLs affect the stability of two-dimensional single-spike bound
states, we develop mathematical theorems for the orbital stability
and instability of two-dimensional single-spike bound states of
(\ref{gp1}) under different conditions of linear and nonlinear
OLs.

To get two-dimensional single-spike bound states of (\ref{gp1}),
we may assume $N=2$, $D>0$ and the scattering length $a$,
i.e.,~$\mu$ is negative and large due to the Feshbach resonance
(cf.~\cite{DCC}). Setting $h^2=D/(-\mu)$,
$V(x)=V_{trap}(x)/(-\mu)$ and suitable time scale, the
equation~(\ref{gp1}) with negative and large $\mu$ can be
equivalent to a semi-classical nonlinear Schr\"{o}dinger
equation~(NLS) given by
\begin{align}\label{id1.1-1} -ih\frac{\partial\psi}{\partial
t}=h^2\Delta\psi-V\,\psi+m\,|\psi|^2\psi,\quad x\in\mathbb{R}^2\,,
t>0\,,
\end{align} where $0<h\ll 1$ is a small parameter, $V=V(x)$ is a
smooth nonnegative function and $m=m(x)$ is a smooth positive
function. For the spatial dimension $N\geq 1$, we may generalize
the equation~(\ref{id1.1-1}) to a NLS having the following form
\begin{align}\label{id1.1}
-ih\frac{\partial\psi}{\partial
t}=h^2\Delta\psi-V\,\psi+m\,|\psi|^{p-1}\psi,\quad
x\in\mathbb{R}^N\,, t>0\,,
\end{align} with critical exponent
\begin{align}\label{id1.2}
p=1+\frac{4}{N}\,,\quad N\geq 1\,.
\end{align} In particular, when $N=2$, the equation (\ref{id1.1}) with (\ref{id1.2}) is exactly
same as (\ref{id1.1-1}).

Single-spike bound states of (\ref{id1.1}) are of the form
$\psi(x,t)=e^{i\lambda\,t/h}u(x)$, where $\lambda$ is a positive
constant and $u=u(x)$ is a positive solution of the following
nonlinear elliptic equation
\begin{align}\label{id1.3}
h^2\Delta u-\left(V+\lambda\right)u+m\,u^p=0\,,\quad u\in
H^1(\mathbb{R}^N)\,,
\end{align} with zero Dirichlet boundary condition,
i.e.,~$u(x)\rightarrow 0$ as $|x|\rightarrow\infty$. When $V\equiv
0$ and $m\equiv 1$, problem~(\ref{id1.3}) admits a unique radially
symmetric ground state which is stable for any $\lambda>0$ if $p
<1+\frac{4}{N}$, and unstable for any $\lambda>0$ if $ p \geq
1+\frac{4}{N}$ (cf.~\cite{[BC]}, \cite{[CL]} and \cite{[We]}). For
$V\not\equiv 0$ or $m\not\equiv 1$, there exists $u_h$ a
single-spike solution of (\ref{id1.3}), provided both $V$ and $m$
are bounded and satisfy another conditions, for example, conditions in the following Theorem \ref{thm1.1}-\ref{thm1.4} (cf.~\cite{[G1]}). For other other nonlinearity in the possibly degenerate setting, see \cite{[abc]}, \cite{[FW]}, \cite{[Gu]}, \cite{[PF1]}, \cite{[PF2]}, \cite{[R1]}, \cite{[WXZ]}, \cite{[WZ]}, \cite{wdirichlet} and reference therein. Hereafter, we set $\psi_h(x,t):=e^{i\lambda t/h}u_h(x)$ as a
single-spike bound state of (\ref{id1.1}), where $u_h$ is the
single-spike solution of (\ref{id1.3}).

In this paper, we want to study the orbital stability of the bound
state $\psi_h$ for the equation~(\ref{id1.1}) with critical
exponent {(\ref{id1.2}). One may regard the bound state $\psi_h$
as an orbit of (\ref{id1.1}). From~\cite{[GS1]}, the orbital
stability of $\psi_h$ is defined as follows: For all $\epsilon>0$,
there exists $\delta>0$ such that if $\|\psi_0-u_h\|_{H^1}<\delta$
and $\psi$ is a solution of (\ref{id1.1}) in some interval
$[0,t_0)$ with $\psi|_{t=0}=\psi_0$, then $\psi(\cdot,t)$ can be
extended to a solution in $0\leq t<\infty$ and
$\sup_{0<t<\infty}\inf_{s\in\mathbb{R}}\|\psi(\cdot,
t)-\psi_h(\cdot,s)\|_{H^1}<\epsilon$. Otherwise, the orbit
$\psi_h$ is called orbital unstable.

The functions $V=V(x)$ and $m=m(x)$ may play a crucial role on the
orbital stability of $\psi_h$. When $m\equiv 1$ and $V$ is of class
$(V)_a$ and fulfills other conditions in \cite{[O1]}-\cite{[O2]},
the orbital stability and instability of $\psi_h$ for the
equation~(\ref{id1.1}) was established by Lin and Wei~\cite{[lw5]}
if $V$ has non-degenerate critical points. Under different
conditions, e.g., $h=1$ and $\lambda$ is large, results of the
orbital stability problem can be found in~\cite{[f1]}. One may also
remark that the orbital stability problem of NLS with inhomogeneous
nonlinearity has been investigated in~\cite{BF} but only for the
subcritical case, i.e.,~$1< p < 1+\frac{4}{N}$.

To state our main results, we need to introduce some notations. It
is well-known that the positive solution of
\begin{equation}\label{160}
\left\{ \begin{aligned}
&\Delta w-w+w^p=0\ \ \mbox{in} \ \ \mathbb{R}^N\,,\\
&w(0)=\max_{y \in \mathbb{R}^N} w(y)\,,\quad w(y)\to 0 \ \mbox{as} \
|y|\to +\infty\,.
\end{aligned} \right.
\end{equation}
is radial~\cite{[GNN]} and unique~\cite{[K]}. We denote the
solution and its linearized operator as $w=w(r)$ and
\begin{equation}\label{161} L_0:=\Delta-1+pw^{p-1}\,,\end{equation} respectively.
For the orbital stability of $\psi_h$, we set
\begin{equation}\label{id1.4}
L_h:=h^2\Delta-\left(V+\lambda\right)+m\,pu_h^{p-1}
\end{equation}
as the linearized operator of (\ref{id1.3}) with respect to $u_h$
and
\begin{align}\label{id1.5}
d(\lambda)= \int_{\mathbb{R}^N}\left[\frac{h^2}{2}|\nabla u_h|^2
+\frac{1}{2}\left(V+\lambda\right)u_h^2-\frac{1}{p+1}m\,u_h^{p+1}\,\right]dx\,,
\end{align}
as the energy of $u_h$. Observe that $u_h$ may depend on the
variable $\lambda$. Assume that $d(\lambda)$ is non-degenerate,
i.e., $ d\,''(\lambda) \not = 0$. Let $p(d\,'')=1$ if $ d\,'' >0$; $
p(d\,'')=0$ if $ d\,'' <0$, and $n(L_h)$ be the number of positive
eigenvalues of $L_h$. According to general theory of orbital
stability of bound states (cf.~\cite{[GS1]}, \cite{[GS2]}), $\psi_h$
is orbital stable if $n(L_h)=p(d\,'')$, and orbital unstable if $
n(L_h)- p(d\,'')$ is odd (see page~309 of~\cite{[GS2]}). It is
remarkable that if both $V$ and $m$ are constant and
$p=1+\frac{4}{N}$, then $d\,''(\lambda)=0$. Consequently, from now
on, we consider the critical exponent $p=1+\frac{4}{N}$ and assume
the point $x_0$ as a non-degenerate critical point of the function
$G$ defined by (cf.~\cite{[G1]}, \cite{[WXZ]})
\begin{equation}\label{191}
G(x):=\big[V(x)+\lambda\big]\,m^{-N/2}(x)\,,\quad\forall
x\in\mathbb{R}^N\,,\end{equation} provided $V\not\equiv 0$ and
$m>0$ in $\mathbb{R}^N$. When $V\equiv 0$ in $\mathbb{R}^N$, $x_0$
is set as a non-degenerate critical point of the function $m$.

For simplicity, we firstly switch off the potential $V$ and obtain
the following result.
\begin{thm}\label{thm1.1}
Let $N\leq3$ be a positive integer, $p=1+\frac{4}{N}$ and the
potential $V\equiv 0$. Assume the function $m=m(x)$ satisfies
\begin{equation}\label{id1.6}
m\in C^4(\mathbb{R}^N);0<m_0\leq m(x)\leq m_1<\infty;|m^{(i)}(x)|\leq
C{\rm{exp}}(\gamma|x|),\,i=1,2,3,4,
\end{equation}
where $m_0,m_1,\gamma$ and $C$ are positive constants, and $m^{(i)}(x)$
are the $i$-th derivatives of $m(x)$. Suppose also that $x_0$ be a non-degenerate critical point of $m(x)$ ($x_0$ is independent of $\lambda$). Let
$\psi_h(x,t):=e^{i\lambda t/h}u_h(x)$ be a bound state of (\ref{id1.1}), where $u_h$ is a single-spike solution of (\ref{id1.3})
concentrating at $x_0$.
Assume also
\begin{eqnarray}\label{con1.1}
m(x_0)\Delta^2 m(x_0)&<&C_{N,1}|\Delta
m(x_0)|^2+C_{N,2}\Big[N\|\nabla^2m(x_0)\|_2^2-|\Delta
m(x_0)|^2\Big]\nonumber\\
&&+C_{N,3}m(x_0)\nabla(\Delta
m)(x_0)\cdot\big[\nabla^2m(x_0)\big]^{-1}\nabla(\Delta m)(x_0)\,,
\end{eqnarray}
where
\begin{align}
&C_{N,1}=\frac{2(N+2)^2\int\limits_0^\infty
r^{N+1}w^pL_0^{-1}\big(r^2w^p\big)dr}{N^2\int\limits_0^\infty
r^{N+3}w^{p+1}dr}\,,\label{id1.9-1} \\
&C_{N,2}=\frac{4(N+2)\int\limits_0^\infty
r^{N+1}w^p\Phi_0dr}{N^2\int\limits_0^\infty
r^{N+3}w^{p+1}dr}\,,\label{id1.9-2}\\
&C_{N,3}=\frac{(N+2)\big(\int\limits_0^\infty
r^{N+1}w^{p+1}dr\big)^2} {N\int\limits_0^\infty
r^{N-1}w^{p+1}dr\int\limits_0^\infty
r^{N+3}w^{p+1}dr},\label{id1.9-3}
\end{align}
are constants depending only on $N$. Here $\Phi_0=\Phi_0(r)$
satisfies
\begin{equation}\label{id1.7}
\left\{ \begin{aligned}
&\Phi_0''+\frac{N-1}{r}\Phi_0'-\Phi_0+pw^{p-1}\Phi_0-\frac{2N}{r^2}\Phi_0-r^2w^p=0,\,r=|x|\in(0,\infty),\\
&\Phi_0(0)=\Phi_0'(0)=0.
\end{aligned} \right.
\end{equation}
where $L_0$ is defined in (\ref{161}). Then for any $\lambda>0$,
$\psi_h$ is orbitally stable if $h$ is sufficiently small and
$x_0$ is a non-degenerate local maximum point of the function $m$.
Furthermore, for any $\lambda>0$, $\psi_h$ is orbitally unstable
if $h$ is sufficiently small and the number of positive
eigenvalues of the Hessian matrix $\nabla^2m(x_0)$ is odd.
\end{thm}
\noindent {\bf Remark 1:} When $N=1$, $x_0=0$ and the function $m$
satisfies $m^{\prime\prime\prime}(x_0)=0\,,$ (see~(C.2)
of~\cite{FSW}), the condition~(\ref{con1.1}) of
Theorem~\ref{thm1.1} is exactly same as the condition~(4.14)
of~\cite{FSW}. For $N\geq 2$, G.Fibich and X.-P.Wang
(cf.~\cite{fw}) considered the function $m$ with radial symmetry,
i.e., $m=m(r), r=|x|$ and $m^{\prime\prime\prime}(0)=0\,,$ and
studied the orbital stability problem only for radial
perturbations. Here we may include the case that the function $m$
is not radially symmetric and the third order derivatives of the
function $m$ at $x_0$ can be nonzero. Moreover, we study the
orbital stability problem for general perturbations including the
non-radial perturbations. Consequently, Theorem~\ref{thm1.1} can
be regarded as the most general theorem on the orbital stability
problem of semiclassical NLS equations with critical exponent and
nonlinear OLs.

When the potential $V$ is turned on, we may generalize the
argument of Theorem~\ref{thm1.1} to obtain three theorems as
follows:
\begin{thm}\label{thm1.2}
Let $N\leq3$ be a positive integer, $p=1+\frac{4}{N}$. Assume both
the potential $V=V(x)$ and the function $m=m(x)$ satisfy the
following conditions: there exist positive constants $V_0,V_1,m_0,m_1,\gamma$ and
$C$ such that
\begin{equation}
V\in C^2(\mathbb{R}^N);0<V_0\leq V(x)\leq V_1<\infty;\quad |V^{(i)}(x)|\leq
C{\rm{exp}}(\gamma|x|),\,i=1,2,
\end{equation}
and
\begin{equation}
m\in C^2(\mathbb{R}^N);0<m_0\leq m(x)\leq m_1<\infty;\quad |m^{(i)}(x)|\leq
C{\rm{exp}}(\gamma|x|),\,i=1,2,
\end{equation}
where $V^{(i)}(x),m^{(i)}(x)$ are the $i$-th derivatives of
$V(x),m(x)$, respectively. Suppose also that $x_0$ be a
non-degenerate critical point of the function~$G$ defined in
(\ref{191}) for fixed $\lambda>0$ ($x_0$ may depend on $\lambda$). Let $\psi_h(x,t):=e^{i\lambda
t/h}u_h(x)$ be a bound state of (\ref{id1.1}), where $u_h$ is a
single-spike solution of (\ref{id1.3}) concentrating at $x_0$. Then $\psi_h$ is orbitally
unstable if $h$ is sufficiently small and $x_0$ is a
non-degenerate local minimum point of $G$ such that $\nabla
V(x_0)\neq0$.
\end{thm}
\begin{thm}\label{thm1.3}
Under the same hypotheses of Theorem~\ref{thm1.2}, assume also that $\nabla
V(x_0)=0$ and $\Delta V(x_0)\neq0$ (thus $x_0$ may be independent of $\lambda$). Let $n$ be the number of
negative eigenvalues of the matrix $ \nabla^2 G(x_0)$. Then $\psi_h$ is orbitally stable if $h$ is
sufficiently small and $x_0$ is a non-degenerate local minimum
point of $G$ with $\Delta V(x_0)>0$. Furthermore, $\psi_h$ is
orbitally unstable if $h$ is sufficiently small and
$n-\frac{1}{2}\left(1+\frac{\Delta V(x_0)}{|\Delta V(x_0)|}\right)
$ is even.
\end{thm}
\begin{thm}\label{thm1.4}
Under the same hypotheses of Theorem~\ref{thm1.2}, assume also that $\nabla
V(x_0)=0$, $\Delta V(x_0)=0$ and (\ref{id1.6}) holds for both $V$
and $m$. Let $n$ be the number of negative eigenvalues of the
matrix $\nabla^2 G(x_0)$. Suppose also that $H(x_0)>0$, where $H(x_0)$
defined in ~(\ref{id4.23}) involves the $i$-th derivatives (for
$0\leq i\leq 4$) of $V$ and $m$ at $x_0$. Then $\psi_h$ is orbitally stable if $h$ is sufficiently small and $x_0$ is a non-degenerate local minimum point of $G$.
Furthermore, $\psi_h$ is orbitally unstable
if $n$ is odd.
\end{thm}
\noindent {\bf Remark 2:} Theorem~\ref{thm1.2}-\ref{thm1.4} may
include all the cases of values $\nabla V(x_0)$ and $\Delta
V(x_0)$ for the orbital stability problem of~(\ref{id1.1}) with
critical exponent~{(\ref{id1.2}). Theorem~\ref{thm1.3} may
generalize the main result of~\cite{[lw5]} to the case that the
function $m$ is a positive and nonconstant function. As $V\equiv
0$, Theorem~\ref{thm1.4} coincides with Theorem~\ref{thm1.1}
because of
\begin{equation*}
\nabla^2G(x_0)=m(x_0)^{-\frac{N}{2}-1}\Big[m(x_0)\nabla^2V(x_0)-\frac{N}{2}\big[V(x_0)
+\lambda\big]\nabla^2m(x_0)\Big]\,.
\end{equation*}

\noindent {\bf Remark 3:} In the following we give examples in dimension $N=2$. Similar examples in dimension $N=1$ and $3$ can also be given. Fist for $x\in\mathbb{R}$ we define
\begin{align*}
&X_1(x)=\sin x+\frac{1}{6}\sin^3x=\frac{9}{8}\sin x-\frac{1}{24}\sin(3x),\\
&X_2(x)=2(1-\cos x)+\frac{1}{3}(1-\cos x)^2=\frac{5}{2}-\frac{8}{3}\cos x+\frac{1}{6}\cos(2x),\\
&X_3(x)=\sin^3x=\frac{3}{4}\sin x-\frac{1}{4}\sin(3x),\\
&X_4(x)=4(1-\cos x)^2=6-8\cos x+2\cos(2x),
\end{align*}
respectively. Then $X_1,X_2,X_3$ and $X_4$ satisfy
\begin{align*}
&|X_1|\leq\frac{7}{6}, X_1'(0)=1, X_1^{(j)}(0)=0,\ \mathrm{for}\ 2\leq j\leq 4,\\
&0\leq X_2\leq\frac{16}{3},X_2''(0)=1, X_2^{(j)}(0)=0,\ \mathrm{for}\ j=1,3,4,\\
&|X_3|\leq1,X_3^{(3)}(0)=1, X_3^{(j)}(0)=0,\ \mathrm{for}\ j=1,2,4,\\
&0\leq X_4\leq16,X_4^{(4)}(0)=1, X_4^{(j)}(0)=0,\ \mathrm{for}\ j=1,2,3.
\end{align*}

Next for $(x,y)\in\mathbb{R}^2$ we set
\begin{align}\label{V}
V(x,y)=a_0+\sum\limits_{i=1}^4a_iX_i(x)++\sum\limits_{i=1}^4b_iX_i(y),
\end{align}
and
\begin{align}\label{m}
m(x,y)=c_0+\sum\limits_{i=1}^4c_iX_i(x)+\sum\limits_{i=1}^4d_iX_i(y),
\end{align}
where $a_i,b_i,c_i,$ and $d_i$ are constants. By the properties of $X_1,X_2,X_3$ and $X_4$, the $i$-th derivatives of $V$ and $m$ at $x_0=(0,0)$ depend only on $a_i,b_i$ and $c_i,d_i$ respectively for $1\leq i\leq4$. Recall that $G(x,y)=[V(x,y)+\lambda]m^{-1}(x,y)$ for $N=2$, we have
\begin{align*}
\nabla G(0)=c_0^{-2}\big(c_0a_1-(a_0+\lambda)c_1,c_0b_1-(a_0+\lambda)d_1\big)^T,
\end{align*}
and then if $\nabla G(0)=0$,
\begin{align*}
\nabla^2G(0)=c_0^{-2}
\begin{pmatrix}
c_0a_2-(a_0+\lambda)c_2&0\\0&c_0b_2-(a_0+\lambda)d_2
\end{pmatrix}.
\end{align*}

Now we can give examples for the potentials $V$ and $m$ which satisfy the assumptions in Theorems 1.2-1.4.

\begin{enumerate}[(I)]
\item (Examples for Theorem 1.2) $N=2,x_0=(0,0)$, $V$ and $m$ given in (\ref{V}) and (\ref{m}) and $a_i,b_i,c_i,d_i$ satisfy
\begin{align*}
&c_0=a_0+\lambda,(a_1,b_1)=(c_1,d_1)\neq0,a_2>c_2>0,b_2>d_2>0,\\
\mathrm{and}&\ c_0>\frac{7}{6}(|a_1|+|b_1|),a_0>\frac{7}{6}(|c_1|+|d_1|),
a_i=b_i=c_i=d_i=0\quad\mathrm{for}\ i=3,4,
\end{align*}

\item (Examples for Theorem 1.3) First a special case for Theorem 1.3 is that $\nabla m(x_0)=0,\nabla^2m(x_0)=0$ and $x_0$ is a non-degenerate critical point of $V(x)$. Here we give another examples. The first one is in the stability case and the second is in the instability case.
    \begin{enumerate}
    \item (Stability) $N=2,x_0=(0,0)$, $V$ and $m$ given in (\ref{V}) and (\ref{m}) and $a_i,b_i,c_i,d_i$ satisfy
        \begin{align*}
        a_0>0,c_0>-\frac{32}{3}c_2>0,c_0>-\frac{32}{3}d_2>0,a_2>0,b_2>0,a_i=b_i=c_i=d_i=0\ \mathrm{for}\ i=1,3,4,
        \end{align*}
        then for any $\lambda>0$, the conditions in Theorem 1.3 for orbital stability will be satisfied.

    \item (Instability) $N=2,x_0=(0,0)$, $V$ and $m$ given in (\ref{V}) and (\ref{m}) and $a_i,b_i,c_i,d_i$ satisfy
        \begin{align*}
        &a_0>-\frac{16}{3}b_2>0,c_0>-\frac{16}{3}c_2>0,a_2+b_2>0,d_2>0,\\
        \mathrm{and}&\ a_i=b_i=c_i=d_i=0\ \mathrm{for}\ i=1,3,4,
        \end{align*}
        then for any $\lambda>0$, the conditions in Theorem 1.3 for orbital instability will be satisfied.
    \end{enumerate}

\item (Examples for Theorem 1.4) Here we give two different examples. First we give examples in the case of $a_4=b_4=0$. Specially, Theorem 1.4 is in this case.
    \begin{enumerate}
    \item (Stability) $N=2,x_0=(0,0)$, $V$ and $m$ given in (\ref{V}) and (\ref{m}) and $a_i,b_i,c_i,d_i$ satisfy
        \begin{align*}
         &a_0>0,c_0>-\frac{32}{3}c_2>0,c_0>-\frac{32}{3}d_2>0,c_0>-32c_4>0,|c_2|,|d_2|\ \mathrm{small},\mathrm{or}\ c_0,|c_4|\ \mathrm{large},\\
         &\mathrm{and}\ a_i=b_i=0=c_3=d_3=d_4\ \mathrm{for}\ i=1,2,3,4,
        \end{align*}
        then for any $\lambda>0$, the conditions in Theorem 1.4 for orbital stability will be satisfied. Here $|c_2|,|d_2|$ small or $c_0,|c_4|$ large are independent on $\lambda$.

    \item (Instability) $N=2,x_0=(0,0)$, $V$ and $m$ given in (\ref{V}) and (\ref{m}) and $a_i,b_i,c_i,d_i$ satisfy
        \begin{align*}
         &a_0>0,c_0>-\frac{16}{3}c_2>0,d_2>0,c_0>-16c_4>0,|c_2|,|d_2|\ \mathrm{small},\mathrm{or}\ c_0,|c_4|\ \mathrm{large},\\
         &\mathrm{and}\ a_i=b_i=0=c_3=d_3=d_4\ \mathrm{for}\ i=1,2,3,4,
        \end{align*}
        then for any $\lambda>0$, the conditions in Theorem 1.4 for orbital instability will be satisfied. Here $|c_2|,|d_2|$ small or $c_0,|c_4|$ large are independent on $\lambda$.
    \end{enumerate}

    Second we give examples in the case of $a_4+b_4\neq0$.
    \begin{enumerate}[1]
    \item (Stability) $N=2,x_0=(0,0)$, $V$ and $m$ given in (\ref{V}) and (\ref{m}) and $a_i,b_i,c_i,d_i$ satisfy
        \begin{align*}
         &a_0>0,c_0>-\frac{32}{3}c_2>0,c_0>-\frac{32}{3}d_2>0,a_4>0,b_4>0,(a_4+b_4)\ \mathrm{large},\\
         &\mathrm{and}\ a_i=b_i=0=c_3=c_4=d_3=d_4\ \mathrm{for}\ i=1,2,3,
        \end{align*}
        then for fixed $\lambda>0$, the conditions in Theorem 1.4 for orbital stability will be satisfied. Here $(a_4+b_4)$ large may depend on $\lambda$.

    \item (Instability) $N=2,x_0=(0,0)$, $V$ and $m$ given in (\ref{V}) and (\ref{m}) and $a_i,b_i,c_i,d_i$ satisfy
        \begin{align*}
         &a_0>0,c_0>-\frac{16}{3}c_2>0,d_2>0,a_4>0,b_4>0,(a_4+b_4)\ \mathrm{large},\\
         &\mathrm{and}\ a_i=b_i=0=c_3=c_4=d_3=d_4\ \mathrm{for}\ i=1,2,3,4,
        \end{align*}
        then for fixed $\lambda>0$, the conditions in Theorem 1.4 for orbital instability will be satisfied. Here $(a_4+b_4)$ large may depend on $\lambda$.
    \end{enumerate}
\end{enumerate}

The rest of this paper is organized as follows: In Section 2, we
show the properties of $u_h$. Then we state the proof of
Theorem~\ref{thm1.1} in Section~3.
Theorem~\ref{thm1.2}-\ref{thm1.4} are proved in Section~4.

\noindent {\bf Acknowledgments:} The research of the first author is
partially supported by a grant from NCTS and NSC of Taiwan. The
research of the second author is partially supported by an Earmarked
Grant from RGC of Hong Kong.

\section{Preliminaries}
\ \ \ \ In this section, we study the properties of $u_h$ a
single-spike bound state of (\ref{id1.3}) concentrated at a
non-degenerate critical point of
$G(x):=\big[V(x)+\lambda\big]\,m^{-N/2}(x)$ (cf.~\cite{[G1]},
\cite{[WXZ]}). Let $x_h$ be the unique local maximum point of
$u_h$. So $x_h\rightarrow x_0$ as $h\to 0$.

Let $v_h(y):=u_h(hy+x_h)$ for all $y\in\mathbb{R}^N$. Then by
(\ref{id1.3}), $v_h$ is a positive solution of
\begin{equation}\label{eq2.1}
\Delta v-\big[V(hy+x_h)+\lambda\big]v+m(hy+x_h)v^p=0.
\end{equation}
For notation convenience, we still denote
\begin{equation}\label{id2.1}
L_h:=\Delta-\big[V(hy+x_h)+\lambda\big]+m(hy+x_h)pv_h^{p-1}
\end{equation}
as the linearized operator of the equation~(\ref{eq2.1}) with
respect to the solution~$v_h$. As the result of~\cite{[WXZ]},
$v_h$ can be written as $v_h=w_{x_h}+\phi_h$, where $w_{x_h}$ is
the unique positive solution of
\begin{equation} \label{eq2.2}
\left\{ \begin{aligned}
&\Delta w-\big[V(x_h)+\lambda\big]w+m(x_h)w^p=0\ \ \mbox{in} \ \ \mathbb{R}^N\,,\\
&w(0)=\max_{y \in \mathbb{R}^N} w(y)\,,\quad w(y)\to 0 \ \mbox{as}
\ |y|\to +\infty\,,
\end{aligned} \right.
\end{equation}
and
\begin{equation}\label{id1.8}
\|\phi_h\|_{\infty}\rightarrow0\quad\hbox{ as }\:h\to 0\,.
\end{equation}
Moreover, \begin{equation}\label{id1.8-1} \quad v_h(y)\leq
C|y|^{\frac{1-N}{2}}{\rm{exp}}\big(-\overline{V}^{1/2}|y|\big)\,,\quad\forall\,y\in\mathbb{R}^N\,,
\end{equation}
where $\overline{V}:=\inf_{\mathbb{R}^N}\big[V(x)+\lambda\big]$.
From~(\ref{eq2.2}), it is easy to check that
\begin{equation}\label{eq2.2-1}
w_{x_h}(y)=\big[V(x_h)+\lambda\big]^{\frac{1}{p-1}}m(x_h)^{-\frac{1}{p-1}}w(\sqrt{V(x_h)+\lambda}y)\,,
\end{equation}
where $w$ is the positive solution of (\ref{160}).

For the single-spike solution of~(\ref{id1.3}), we recall the
following result from~\cite{[WX]} and \cite{[WXZ]}:
\begin{lem}\label{lem2.1}
Assume that there are positive constants $\gamma$ and $C$ such that
\begin{equation}\label{id2.2}
|\nabla V(x)|,|\nabla m(x)|\leq
C{\rm{exp}}(\gamma|x|)\,,\quad\forall\,x\in\mathbb{R}^N\,.
\end{equation}
 Then
\begin{equation} \label{id2.3}
\int\limits_{\mathbb{R}^N}\Big[\frac{1}{p+1}\nabla
m(hy+x_h)v_h^{p+1}-\frac{1}{2}\nabla V(hy+x_h)v_h^2\Big]dy=0\
\end{equation}
 for $0<h<h_0$, where $h_0$ is a positive constant depending on
$\gamma$ and $\lambda$.
\end{lem}

In the rest of this section, for simplicity, we switch off the
potential $V$, i.e., set $V\equiv 0$. Then by Lemma~\ref{lem2.1}, we
obtain the uniqueness of $u_h$ as follows:
\begin{lem}\label{lem2.2}
Suppose (\ref{id2.2}) holds, $V\equiv 0$ and $x_0$ is a
non-degenerate critical point of $m$. Then $u_h$ is unique.
\end{lem}
\begin{proof}
Suppose $u_{h,1}$ and $u_{h,2}$ are different single-spike
solutions of~(\ref{id1.3}) concentrating at the same point $x_0$.
Let $v_1(y):=u_{h,1}(hy+x_0)$ and $v_2(y):=u_{h,2}(hy+x_0)$. Then
both $v_1$ and $v_2$ satisfy
$$\Delta v-\lambda v+m(hy+x_0)v^p=0\,,\quad\hbox{ for }\:y\in\mathbb{R}^N\,,$$
and $v_1,v_2\rightarrow w_{x_0}$ uniformly on $\mathbb{R}^N$ as
$h\to 0$. Due to $v_1\not\equiv v_2$, we may set
$$\widetilde{v}_h:=\frac{v_1-v_2}{\|v_1-v_2\|}_\infty\,,$$
and then $\widetilde{v}_h$ satisfies
\begin{equation}\label{eq2.3}
\Delta
\widetilde{v}_h-\lambda\widetilde{v}_h+m(x_0)pw_{x_0}^{p-1}\widetilde{v}_h
+[m(hy+x_0)-m(x_0)]pw_{x_0}^{p-1}\widetilde{v}_h+N(\widetilde{v}_h)=0,
\end{equation}
where
$N(\widetilde{v}_h)=m(hy+x_0)\big[v_1^p-v_2^p-pw_{x_0}^{p-1}(v_1-v_2)\big]
/\|v_1-v_2\|_\infty$. Hence by the standard elliptic PDE theorems on
the equation~(\ref{eq2.3}), we may take a subsequence
$\widetilde{v}_h\rightarrow \widetilde{v}_0$, where
$\widetilde{v}_0$ solves
$$\Delta \widetilde{v}_0-\widetilde{v}_0+m(x_0)pw_{x_0}^{p-1}\widetilde{v}_0=0.$$
Consequently, there exist constants $c_j$'s such that
\begin{equation}\label{id2.4}
\widetilde{v}_0=\sum\limits_{j=1}^N\,c_j\partial_jw_{x_0}\,.
\end{equation}

Let $y_h$ be such that
$\widetilde{v}_h(y_h)=\|\widetilde{v}_h\|_\infty=1$ (the same proof
applies if $\widetilde{v}_h(y_h)=-1$). Then by the Maximum
Principle, we have $|y_h|\leq C$. On the other hand,
as~(\ref{id2.3}), we may obtain
$$\int\limits_{\mathbb{R}^N}\nabla
m(hy+x_0)v_1^{p+1}dy=0=\int\limits_{\mathbb{R}^N}\nabla
m(hy+x_0)v_2^{p+1}dy.$$ Thus
\begin{equation}\label{id2.5}
\int\limits_{\mathbb{R}^N}\nabla
m(hy+x_0)\Big(\frac{v_1^{p+1}-v_2^{p+1}}{v_1-v_2}\Big)\widetilde{v}_hdy=0.
\end{equation}
Note that for all $i=1,\cdots,N$, as $h\rightarrow0$,
\begin{eqnarray*}
\partial_im(hy+x_0)=h\sum_{k=1}^{N}\partial_{ik}m(x_0)y_k+\mathrm{o}(h)\,,{\rm{and}}\,\,\,
\frac{v_1^{p+1}-v_2^{p+1}}{v_1-v_2}=(p+1)w_{x_0}^{p}+\mathrm{o}(1)\,.
\end{eqnarray*}
Hence from~(\ref{id2.4}) and (\ref{id2.5}), we may obtain
\begin{equation}
\begin{split}
0&=\int\limits_{\mathbb{R}^N}\Big[h\sum_{k=1}^{N}\partial_{ik}m(x_0)
y_k\Big](p+1)w_{x_0}^p\Big(\sum\limits_{j=1}^Nc_j\partial_j
w_{x_0}\Big)dy
+\mathrm{o}(h)\\
&=-h\sum\limits_{j=1}^N\partial_{ij}m(x_0)c_j\int\limits_{\mathbb{R}^N}w_{x_0}^{p+1}\,dy+\mathrm{o}(h)\,.
 \end{split}\nonumber
 \end{equation}
Hence by the assumption that $\nabla^2m(x_0)$ is non-degenerate,
$c_j=0$ for $j=1,\cdots,N$, i.e.,~$\widetilde{v}_0\equiv0$. This may
contradict to the fact that $1=\widetilde{v}_h(y_h)\rightarrow
\widetilde{v}_0(y_0)$ for some $y_0\in\mathbb{R}^N$. Therefore, we
may complete the proof of Lemma~\ref{lem2.2}.
\end{proof}

By Lemma~\ref{lem2.1}, we may simplify the proof of~\cite{[G]} and
get a shorter proof of the asymptotic behavior of $x_h$'s as
follows:
\begin{lem}\label{lem2.3}
Under the same hypotheses of Lemma~\ref{lem2.2},
\begin{equation}\label{eq2.4}
x_h=x_0+\mathrm{o}(h)\quad\hbox{ as }\:h\to 0\,.
\end{equation}
\end{lem}
\begin{proof}
Fix $i\in\{1,\cdots,N\}$ arbitrarily. By Taylor's expansion of
$\partial_i m(x)$ and $\nabla m(x_0)=0$, we obtain
$$\partial_i
m(hy+x_h)=\sum_{j=1}^{N}\partial_{ij}m(x_0)
(hy_j+x_{h,j}-x_{0,j})+\mathrm{o}(h)+\mathrm{o}(|x_h-x_0|).$$
Hence by Lemma~\ref{lem2.1} and $v_h=w_{x_0}+\mathrm{o}(1)$, we
have
\begin{equation*}
\begin{split}
0&=\int\limits_{\mathbb{R}^N}\partial_i
m(hy+x_h)v_h^{p+1}dy \\
&=
\sum\limits_{j=1}^N\partial_{ij}m(x_0)(x_{h,j}-x_{0,j})\int\limits_{\mathbb{R}^N}w_{x_0}^{p+1}dy
+\mathrm{o}(h)+\mathrm{o}(|x_h-x_0|)
  \end{split}
 \end{equation*}
Here we have used the fact that
$\int\limits_{\mathbb{R}^N}y_jw_{x_0}^{p+1}\,dy=0$ for
$j=1,\cdots,N$. Using the assumption that $\nabla^2m(x_0)$ is
non-degenerate, we obtain~(\ref{eq2.4}).
\end{proof}

Following the idea of~\cite{[lw5]}, we may use Lemma~\ref{lem2.3} to
show the asymptotic behavior of $v_h$ as follows:
\begin{lem}\label{lem2.4} Under the same hypotheses of Lemma~\ref{lem2.2},
\begin{equation}\label{eq2.5}
v_h=w_{x_h}+h^2\phi_2+\mathrm{o}(h^2)\,,\quad\hbox{ as
}\:h\rightarrow 0\,,
\end{equation}
where $\phi_2$ satisfies
\begin{equation}\label{eq2.6}
\Delta\phi_2-\lambda\phi_2+m(x_h)pw_{x_h}^{p-1}\phi_2
+\frac{1}{2}\sum_{i,j=1}^N\partial_{ij}m(x_0)y_iy_jw_{x_h}^p=0\,,{\rm{and}}\
\nabla\phi_2(0)=0.
\end{equation}
\end{lem}
\begin{proof}
Let $\phi_h=v_h-w_{x_h}$. Then it is easy to check that
$|\phi_h|\rightarrow0$ uniformly, and $\phi_h$ satisfies
\begin{equation}\label{eq2.7}
\Delta\phi_h-\lambda\phi_h+m(hy+x_h)pw_{x_h}^{p-1}\phi_h+N(\phi_h)+R(\phi_h)=0\,,{\rm{and}}\
\nabla\phi_h(0)=0,
\end{equation}
where $$N(\phi_h)= m(hy+x_h)\Big[(w_{x_h}+\phi_h)^p- w_{x_h}^p - p
w_{x_h}^{p-1} \phi_h\Big],$$ and
$$R(\phi_h)=\Big[m(hy+x_h)-m(x_h)\Big]w_{x_h}^p.$$ Note that by
Lemma~\ref{lem2.3} and $\nabla m(x_0)=0$,
\begin{align}\label{id2.6}
m(hy+x_h)-m(x_h)
=&hy\cdot\nabla m(x_h)+\frac{h^2}{2}\sum_{i,j=1}^N
\partial_{ij}m(x_h)y_iy_j+\mathrm{o}(h^2)\notag\\
=&\frac{h^2}{2}\sum_{i,j=1}^N\partial_{ij}m(x_0)y_iy_j+\mathrm{o}(h^2).
\end{align}

Now we claim that $|\phi_h|\leq c\,h^2$ by contradiction. Suppose
that $h^{-2}\|\phi_h\|_{\infty}\rightarrow\infty$. Let
$\widetilde{\phi}_h=\phi_h/\|\phi_h\|_\infty$. Then
$\widetilde{\phi}_h$ satisfies
\begin{align}\label{eq2.8}
\Delta\widetilde{\phi}_h-\lambda\widetilde{\phi}_h+m(hy+x_h)pw_{x_h}^{p-1}\widetilde{\phi}_h
+\frac{N(\phi_h)}{\|\phi_h\|_{\infty}}+\frac{R(\phi_h)}{\|\phi_h\|_{\infty}}=0\,.
\end{align}
Note that by (\ref{id2.6}),
\begin{align}\label{id2.7}
\frac{R(\phi_h)}{\|\phi_h\|_{\infty}}\leq
C\frac{h^2}{\|\,\phi_h\|_{\infty}}\,.
\end{align}
Let $y_h$ be such that
$\widetilde{\phi}_h(y_h)=\|\widetilde{\phi}_h\|_\infty=1$ (the same
proof applies if $\widetilde{\phi}_h(y_h)=-1$). Then by
(\ref{eq2.8})$-$(\ref{id2.7}) and the Maximum Principle, we have
$|y_h|\leq C$. On the other hand, by the usual elliptic regularity
theory, we may take a subsequence
$\widetilde{\phi}_h\rightarrow\widetilde{\phi}_0$, where
$\widetilde{\phi}_0$ satisfies
\begin{align}
\Delta\widetilde{\phi}_0-\lambda\widetilde{\phi}_0+m(x_0)pw_{x_0}^{p-1}\widetilde{\phi}_0=0\,,{\rm{and}}\
\nabla\widetilde{\phi}_0(0)=0. \nonumber
\end{align} Hence $\widetilde{\phi}_0\equiv 0$. This
may contradict to the fact that
$1=\widetilde{\phi}_h(y_h)\rightarrow\widetilde{\phi}_0(y_0)$ for
some $y_0$. Therefore, we may complete the claim that
$|\,\phi_h|\leq c\,h^2$.

Now we set $\phi_{h,2}=\phi_h-h^2\phi_2$. Then
$\phi_{h,2}=\mathrm{O}(h^2)$ and satisfies
$$\Delta\phi_{h,2}-\lambda\phi_{h,2}+m(hy+x_h)pw_{x_h}^{p-1}\phi_{h,2}
+N(\phi_{h,2})+R(\phi_{h,2})=0\,,{\rm{and}}\,\nabla\phi_{h,2}(0)=0$$
where
$$N(\phi_{h,2})=m(hy+x_h)\Big[(w_{x_h}+h^2\phi_2+\phi_{h,2})^p-w_{x_h}^p
-pw_{x_h}^{p-1}(h^2\phi_2+\phi_{h,2})\Big],$$ and
$$R(\phi_{h,2})=\Big[m(hy+x_h)-m(x_h)-\frac{h^2}{2}\sum\limits_{i,j=1}^N
\partial_{ij}m(x_0)y_iy_j\Big]w_{x_h}^p+h^2\Big[m(hy+x_h)-m(x_h)\Big]pw_{x_h}^{p-1}\phi_2.$$
Thus as for previous argument, we may have
$\phi_{h,2}=\mathrm{o}(h^2)$ and complete the proof of
Lemma~\ref{lem2.4}.
\end{proof}

As for Proposition 3.1 of~\cite{[KW]}, one may get two lemmas as
follows:
\begin{lem}\label{lem2.5}
For $h$ small enough, the maps
\begin{align*}
L_{x_h}\phi:=\Delta
\phi-\big[V(x_h)+\lambda\big]\phi+m(x_h)pw^{p-1}_{x_h}\phi
\end{align*}
are uniformly invertible from $K_{x_h}^{\perp}$ to
$C_{x_h}^{\perp}$, where
\begin{align}
K_{x_h}^{\perp}=\left\{\phi\in
H^2(\mathbb{R}^N)\left|\int_{\mathbb{R}^N}\phi\partial_j
w_{x_h}dy=0\,,
j=1,\cdots,N\right.\right\}&\subset H^2(\mathbb{R}^N),\notag\\
C_{x_h}^{\perp}=\left\{\phi\in
L^2(\mathbb{R}^N)\left|\int_{\mathbb{R}^N}\phi\partial_j
w_{x_h}dy=0\,, j=1,\cdots,N\right.\right\}&\subset
L^2(\mathbb{R}^N).\notag
\end{align}
\end{lem}
\begin{lem}\label{lem2.6}
The map
\begin{align}
L_{x_0}\phi:=\Delta
\phi-\big[V(x_0)+\lambda\big]\phi+m(x_0)pw_{x_0}^{p-1}\phi\nonumber
\end{align}
has eigenvalues $\mu_j\,, j=1,\cdots,N+2$ satisfying
  $$\mu_1>0=\mu_2=\cdots=\mu_{N+1}>\mu_{N+2}\geq\cdots\,,$$
 where the kernel of $L_{x_0}$ is spanned by $\partial_jw_{x_0}$, $j=1,\cdots,N$ and $\mu_1$ is simple.
\end{lem}

In this section, our main result is the small eigenvalue estimates
of $L_h$ given by
\begin{thm}\label{thm2.7}
Under the same hypotheses of Lemma~\ref{lem2.2}, for $h$ small
enough, the eigenvalue problem
\begin{align}\label{eq2.9}
L_h\varphi_h =\mu_h\varphi_h
\end{align}
has exactly $N$ eigenvalues $\mu_h^j\,, j=1,\cdots,N$, in the interval $[\frac{1}{2}\mu_1,\frac{1}{2}\mu_{N+2}]$, which satisfy
\begin{align}\label{id2.8} \frac{\mu_h^j}{h^2}\rightarrow
c_0\nu_{j}\,,\:(\hbox{up\ to\ a\ subsequence})\quad\hbox{ as }\:h\to
0\,,\quad\hbox{ for }\: j=1,\cdots, N,
\end{align}
where $\mu_1$ and $\mu_{N+2}$ are defined in Lemma~\ref{lem2.6},
$\nu_j$'s are the eigenvalues of the Hessian matrix $\nabla^2
m(x_0)$ and $c_0=\frac{N}{2m(x_0)}$ is a positive constant.
Furthermore, the corresponding eigenfunctions $\varphi_h^j$'s
satisfy
\begin{equation}
\varphi_h^j=\sum\limits_{i=1}^N\big[a_{ij}+\mathrm{o}(1)\big]\partial_{i}
w_{x_h}+\mathrm{O}(h^2)\,,\quad j=1,\cdots,N\,,
\end{equation}
where $\textbf{a}_j=(a_{1j},\cdots,a_{Nj})^T$ is the eigenvector
associated with $\nu_j$, namely,
\begin{equation}
\nabla^2 m(x_0)\textbf{a}_j=\nu_j\textbf{a}_j\,.\end{equation}
Here $o(1)$ is a small quantity tending to zero and $O(1)$ is a
bounded quantity as $h$ goes to zero.
\end{thm}

\noindent {\bf Remark 4:} (1) Since $L_h$ converges to $L_{x_0}$ in the strong resolvent sense, in the interval $(\frac{1}{2}\mu_1,\infty)$ $L_h$ has only one positive eigenvalues $\mu_h^0$, which is simple and goes to $\mu_1$ as $h$ goes to $0$.

(2) After changing variables $t\mapsto t/h,y=(x-x_0)/h$, $L_h$ becomes $-R_h$, which is the notation used in page 190 of \cite{[GS1]}. Thus the number of negative eigenvalues of $R_h$ equals the number of positive eigenvalues of $L_h$, which we denote by $n(L_h)$.

(3) By (\ref{id2.8}), the sign of small eigenvalue $\mu_h^j$ of $L_h$ is the same as the one of eigenvalue $\nu_j$ of $\nabla^2m(x_0)$. If we denote the number of positive eigenvalues of $\nabla^2m(x_0)$ by $n$, then the number of positive eigenvalues of $L_h$ in the interval $[\frac{1}{2}\mu_1,\frac{1}{2}\mu_{N+2}]$ equals $n$. Adding another one in the interval $(\frac{1}{2}\mu_1,\infty)$, the number of positive eigenvalues of $L_h$ equals $n+1$. In particular, if $\nabla^2m(x_0)$ is negative definite, then $n=0$ and thus $n(L_h)=1$.

\begin{proof} We may follow the arguments given in Section~5 of~\cite{[W]}.
Assume that $\|\varphi_h\|_{L^2}=1$. By Lemma (\ref{lem2.6}) it is easy to see that
$\mu_{h}\rightarrow 0$ as $h\to 0$, where
$\mu_{h}\in\{\mu_{h}^1,\cdots,\mu_{h}^N\}$. Then the corresponding
eigenfunctions $\varphi_h$'s can be written as
\begin{align}\label{id2.9}
\varphi_h=\sum_{j=1}^Na_h^j\partial_j w_{x_h}+\varphi_h^\perp,
\end{align}
where $\varphi_h^\perp\in K_{x_h}^\perp$. Hence by (\ref{eq2.9})
and (\ref{id2.9}), $\varphi_h^\perp$ satisfies
\begin{equation}\label{eq2.10}\begin{array}{lll}
&\Delta\varphi_h^\perp-\lambda\varphi_h^\perp+m(x_h)pw_{x_h}^{p-1}\varphi_h^\perp+R(\varphi_h^\perp)
+\sum\limits_{j=1}^Na_h^jL_h\partial_j
w_{x_h}=\mu_h\left(\ds\sum_{j=1}^Na_h^j\partial_j
w_{x_h}+\varphi_h^\perp\right)\,,
\end{array}
\end{equation}
where
$$R(\varphi_h^\perp)=m(hy+x_h)p(v_h^{p-1}-w_{x_h}^{p-1})\varphi_h^\perp
+\Big[m(hy+x_h)-m(x_h)\Big]pw_{x_h}^{p-1}\varphi_h^\perp.$$
Using~(\ref{id2.6}) and Lemma~\ref{lem2.4}, we have
\begin{align}\label{id2.10}
L_h\partial_j w_{x_h} =m(hy+x_h)p(v_h^{p-1}-w_{x_h}^{p-1})\partial_j
w_{x_h}+\Big[m(hy+x_h)-m(x_h)\Big]pw_{x_h}^{p-1}\partial_j
w_{x_h}=\mathrm{O}(h^2).
\end{align}
From Lemma~\ref{lem2.5}, the map
$L_{x_h}=\Delta-\lambda+m(x_h)pw_{x_h}^{p-1}$ is uniformly
invertible in the space $K_{x_h}^\perp$. Thus by (\ref{id2.10}) and
$\mu_h\rightarrow0$, we have
\begin{align}\label{id2.11}
\|\varphi_h^\perp\|_{H^2}\leq c(h^2+|\mu_h|)\sum_{j=1}^N\,|a_h^j|\,.
\end{align}

To estimate $\mu_h$ and $a_h^j$'s, multiplying (\ref{eq2.10}) by
$\partial_k w_{x_h}$ and integrating over $\mathbb{R}^N$, we may
obtain
\begin{align}\label{eq2.11}
\int_{\mathbb{R}^N}\left(L_h\varphi_h^\perp\right)\partial_k
w_{x_h}dy +\sum_{j=1}^Na_h^j\int_{\mathbb{R}^N}\left(L_h\partial_j
w_{x_h}\right)\partial_k w_{x_h}dy
=\mu_h\sum_{j=1}^Na_h^j\int_{\mathbb{R}^N}\partial_j
w_{x_h}\partial_k w_{x_h}dy\,.
\end{align} Here we have used the fact that $\varphi_h^\perp\in
K_{x_h}^\perp$. Using (\ref{id2.10}), (\ref{id2.11}),
$\mu_h=\mathrm{o}(1)$ and integration by parts, we obtain
\begin{align}\label{id2.12}
\int_{\mathbb{R}^N}\left(L_h\varphi_h^\perp\right)\partial_k
w_{x_h}dy =\int_{\mathbb{R}^N}\varphi_h^\perp L_h\partial_k
w_{x_h}dy=\mathrm{o}(h^2)\,,
\end{align}
and
\begin{align}\label{id2.13}
\int_{\mathbb{R}^N}\left(L_h\partial_j w_{x_h}\right)\partial_k
w_{x_h}dy=\frac{h^2}{p+1}\int_{\mathbb{R}^N}w_{x_h}^{p+1}dy\partial
_{jk}m(x_0)+\mathrm{o}(h^2)\,,
\end{align}
which we have proved in Appendix A. Substituting~(\ref{id2.12})
and~(\ref{id2.13}) into~(\ref{eq2.11}), we may obtain
\begin{equation*}
\frac{1}{p+1}\int_{\mathbb{R}^N}w_{x_h}^{p+1}dy\sum\limits_{j=1}^N\partial_{jk}m(x_0)a_h^j
=\frac{\mu_h}{h^2}a_h^k\int\limits_{\mathbb{R}^N}\big(\partial_k
w_{x_h}\big)^2dy+\mathrm{o}(1).
\end{equation*}
Since $\|\varphi_h\|_{L^2}=1$, (\ref{id2.9}) implies that
$\textbf{a}_h:=(a_h^1,\cdots,a_h^N)^T$ is bound. Moreover, by
(\ref{id2.11}), $\textbf{a}_h$ does not converge to $0$. Thus
$\frac{\mu_h^j}{h^{2}}\rightarrow c_0\nu_j$ for $j=1,\cdots,N$ and
$\textbf{a}_h\rightarrow\textbf{a}_j$ , where
$$c_0=\frac{N\int_{\mathbb{R}^N}w_{x_0}^{p+1}dy}{(p+1)\int_{\mathbb{R}^N}|\nabla
w_{x_0}|^2dy}=\frac{N}{2m(x_0)},$$ and $\textbf{a}_j$ is the
eigenvector corresponding to $\nu_j$. Here we have use the fact that
\begin{align*}
\int_{\mathbb{R}^N}|\nabla
w_{x_0}|^2dy=\frac{N}{N+2}m(x_0)\int_{\mathbb{R}^N}w_{x_0}^{p+1}dy\,,
\end{align*}
which can be proved by Pohozeve identity. The rest of the proof
follows from a perturbation result, similar to page 1473-1474
of~\cite{[W]}. We may omit the details here.
\end{proof}

\section{Proof of Theorem~\ref{thm1.1}}
\ \ \ \ \ \ In this Section, we firstly study the asymptotic
expansion of $d\,''(\lambda)$ as $h\rightarrow 0$, and then
complete the proof of Theorem~\ref{thm1.1}. To drive the
$\mathrm{O}(h^4)$ order terms of $d\,''(\lambda)/h^N$, we need the
following lemma:
\begin{lem}\label{lem3.1}
Under the same hypotheses of Lemma~\ref{lem2.2},
\begin{align}\label{id3.1}
x_h=x_0+h^2\textbf{x}_1+\mathrm{O}(h^3)\,,\quad\hbox{ as }\:
h\rightarrow 0\,,
\end{align}
where $\textbf{x}_1\in\mathbb{R}^N$ satisfies
\begin{equation}\label{id3.2}
\nabla^2m(x_0)\textbf{x}_1=-\frac{\int\limits_{\mathbb{R}^N}|y|^2w^{p+1}\,dy}
{2N\lambda\int\limits_{\mathbb{R}^N}w^{p+1}\,dy}\nabla(\Delta
m)(x_0).
\end{equation}
\end{lem}
\begin{proof}
By Lemma~\ref{lem2.3} and $\nabla m(x_0)=0$, for all $i=1,\cdots,N$,
we have
\begin{equation}\label{id3.3}
\partial_i m(hy+x_h)=\sum\limits_{j=1}^N\partial_{ij}m(x_0)\big(hy_j+x_{h,j}-x_{0,j}\big)+\mathrm{O}(h^2).
\end{equation}
 Then by
(\ref{id2.3}), (\ref{id3.3}) and Lemma~\ref{lem2.4}, we have
\begin{equation}
\begin{split}
0&=\int\limits_{\mathbb{R}^N}\partial_i m(hy+x_h)v_h^{p+1}dy\\
 &=\sum\limits_{j=1}^N\partial_{ij}m(x_0)
 \int\limits_{\mathbb{R}^N}\big(hy_j+x_{h,j}-x_{0,j}\big)
 \Big[w_{x_h}^{p+1}+\mathrm{O}(h)\Big]dy+\mathrm{O}(h^2) \\
&=\sum\limits_{j=1}^N\partial_{ij}m(x_0)\big(x_{h,j}-x_{0,j}\big)\int\limits_{\mathbb{R}^N}w_{x_0}^{p+1}dy
+ \mathrm{O}(h^2)\,.
  \end{split}\nonumber
 \end{equation}
Here we have used the fact that
$\int\limits_{\mathbb{R}^N}y_jw_{x_h}^{p+1}\,dy=0$ for
$j=1,\cdots,N$. Thus $x_h=x_0+\mathrm{O}(h^2)$. Consequently, we may
set $x_h=x_0+h^2\overline{x}_h$. Then $\overline{x}_h=\mathrm{O}(1)$
and by Taylor's formula of $\partial_i m(x)$, we have
\begin{eqnarray}\label{id3.4}
\partial_i m(hy+x_h)=\sum\limits_{j=1}^N\partial_{ij}m(x_0)\big(hy_j+h^2\overline{x}_h\big)
+\frac{h^2}{2}\sum\limits_{j,k=1}^N\partial_{ijk}
m(x_0)y_jy_k+\mathrm{O}(h^3).
\end{eqnarray}
Hence by(\ref{id2.3}), (\ref{id3.4}) and Lemma~\ref{lem2.4}, we may
obtain
\begin{align*}
0=&h^2\sum\limits_{j=1}^N\partial_{ij}
m(x_0)\overline{x}_{h,j}\int\limits_{\mathbb{R}^N}w_{x_h}^{p+1}dy
+\frac{h^2}{2}\sum\limits_{j,k=1}^N\partial_{ijk}
m(x_0)\int\limits_{\mathbb{R}^N}y_jy_kw_{x_h}^{p+1}dy
+\mathrm{O}(h^3)\nonumber\\
=&h^2\sum\limits_{j=1}^N\partial_{ij}m(x_0)\overline{x}_{h,j}\int\limits_{\mathbb{R}^N}w_{x_0}^{p+1}dy
+\frac{h^2}{2N}\sum\limits_{k=1}^N\partial_{ikk}
m(x_0)\int\limits_{\mathbb{R}^N}|y|^2w_{x_0}^{p+1}dy
+\mathrm{O}(h^3).
\end{align*}
Here we have used the fact that
\begin{align*}
\begin{cases}
&\int\limits_{\mathbb{R}^N}y_jw_{x_0}^{p+1}dy=0\,,\quad\forall j=1,\cdots,N,\\
&\int\limits_{\mathbb{R}^N}y_jy_kw_{x_0}^{p+1}=\frac{\delta_{jk}}{N}
\int\limits_{\mathbb{R}^N}|y|^2w_{x_0}^{p+1}dy\,,\quad\forall
j,k=1,\cdots,N.
\end{cases}
\end{align*}
Therefore, we may complete the proof because
$$w_{x_0}(y)=\lambda^{N/4}m(x_0)^{-N/4}w(\sqrt\lambda y)\,.$$
\end{proof}

From Lemma~\ref{lem2.4} and~\ref{lem3.1}, we may deduce that
\begin{thm}\label{thm3.2}
Under the same hypotheses of Lemma~\ref{lem2.2}, for $h$ small
enough, $u_h$ is smooth on $\lambda$. Let $R_h:=\frac{\partial
u_h}{\partial\lambda}(hy+x_h)$. Then
\begin{equation}\label{eq3.4}
L_hR_h-v_h=0.
\end{equation}
and
\begin{align}\label{id3.7}
R_h=R_0+\sum\limits_{j=1}^Nc_h^j\partial_j
w_{x_h}+h^2R_1+R_h^\perp\,,
\end{align}
where $R_0=\lambda^{-1}\big(\frac{1}{p-1}v_h+\frac{1}{2}y\cdot\nabla
v_h\big)$, $c_h^j=\mathrm{O}(h)$, $R_h^\perp=\mathrm{O}(h^3)$ and
$R_1$ satisfies
\begin{equation}\label{eq3.2}
\Delta R_1-\lambda
R_1+m(x_h)pw_{x_h}^{p-1}R_1-\frac{1}{2\lambda}\sum\limits_{i,j=1}^N\partial_{ij}m(x_0)y_iy_jw_{x_h}^p=0\,.
\end{equation}
Furthermore,
\begin{equation}\label{id3.9}
\nabla^2m(x_0)\big(h^{-1}\textbf{c}_h\big)\rightarrow-\frac{\int\limits_{\mathbb{R}^N}|y|^2w^{p+1}\,dy}
{2N\lambda^2\int\limits_{\mathbb{R}^N}w^{p+1}\,dy}\nabla(\Delta
m)(x_0)\,,\quad\hbox{ as }\:h\rightarrow0\,,
\end{equation} where $\textbf{c}_h:=(c_h^1,\cdots,c_h^N)^T$.
\end{thm}
\begin{proof}
By Lemma~\ref{lem2.2} and Theorem~\ref{thm2.7}, $u_h$ is unique
and non-degenerate. Consequently, $u_h$ is smooth on $\lambda$ and
$R_h$ satisfies~(\ref{eq3.4}). Now we decompose $R_h$ as
\begin{align}\notag
R_h=R_0+\sum\limits_{j=1}^Nc_h^j\partial_j
w_{x_h}+h^2R_1+R_h^\perp\,,
\end{align}
where $R_h^\perp\in K_{x_h}^\perp$. Then $R_h^\perp$ satisfies
\begin{align}\label{id3.10}
L_hR_h^\perp+\big[L_hR_0+h^2L_hR_1-v_h\big]+\sum_{j=1}^N
c_h^jL_h\partial_j w_{x_h}=0.
\end{align}
As for the proof of Theorem~\ref{thm2.7}, we have
\begin{equation}\label{id3.11}
\|R_h^\perp\|_{H^2}\leq
c\Big(\|L_hR_0+h^2L_hR_1-v_h\|_{L^2}+\sum_{j=1}^N|c_h^j|h^2\Big)\,.
\end{equation}
It is easy to check
\begin{equation}\label{eq3.8}
L_hR_0=v_h-\frac{h}{2\lambda}y\cdot\nabla m(hy+x_h)v_h^p.
\end{equation}
Hence by Lemma~\ref{lem2.4}, \ref{lem3.1}, (\ref{eq3.2})
and~(\ref{eq3.8}), we obtain
\begin{align}\label{id3.16}
&L_hR_0+h^2L_hR_1-v_h\nonumber\\
=&-\frac{h^3}{2\lambda}\Big[\sum\limits_{i,j=1}^N\partial_{ij}m(x_0)x_{1,i}y_j
+\frac{1}{2}\sum\limits_{i,j,k=1}\partial_{ijk}m(x_0)y_iy_jy_k\Big]w_{x_h}^p+\mathrm{O}(h^4)\,.
\end{align}
Consequently, by~(\ref{id3.11}),
\begin{equation}\label{id3.12}
\|R_h^\perp\|_{H^2}\leq c\Big(h^3+\sum_{j=1}^N|c_h^j|h^2\Big)\,.
\end{equation}

To estimate $c_h^j$'s, we may multiply (\ref{id3.10}) by $\partial_k
w_{x_h}$ and integrate over $\mathbb{R}^N$. Then
\begin{align}\label{id3.13}
&\int_{\mathbb{R}^N}(L_hR_h^\perp)\partial_k w_{x_h}dy
+\int_{\mathbb{R}^N}\Big[L_hR_0+h^2L_hR_1-v_h\Big]\partial_k
w_{x_h}dy\nonumber\\
&+\sum\limits_{j=1}^Nc_h^j\int_{\mathbb{R}^N}(L_h\partial_j
w_{x_h})\partial_k w_{x_h}dy=0.
\end{align}
Hence by (\ref{id2.13}), (\ref{id3.13}) may imply
\begin{equation}\label{id3.14}
|c_h^j|\leq\frac{C}{h^2}\left[\big|\int_{\mathbb{R}^N}(L_hR_h^\perp)\partial_k
w_{x_h}dy\big|
+\big|\int_{\mathbb{R}^N}\big[L_hR_0+h^2L_hR_1-v_h\big]\partial_k
w_{x_h}dy\big|\right].
\end{equation}
Using integration by parts and (\ref{id2.10}), we have
\begin{align}\label{id3.15}
\int_{\mathbb{R}^N}(L_hR_h^\perp)\partial_k w_{x_h}dy =
\int_{\mathbb{R}^N}R_h^\perp L_h\partial_k w_{x_h}dy
=\|R_h^\perp\|_{L^2}\mathrm{O}(h^2)\,.
\end{align}
Therefore, by (\ref{id3.16}), (\ref{id3.12}), (\ref{id3.14}) and
(\ref{id3.15}), we may obtain $|c_h^j|=\mathrm{O}(h)$. Consequently,
by~(\ref{id3.12}), $R_h^\perp=\mathrm{O}(h^3)$. Thus
by~(\ref{id3.15}),
\begin{equation}\label{id3.17}
\int_{\mathbb{R}^N}(L_hR_h^\perp)\partial_k
w_{x_h}dy=\mathrm{O}(h^5).
\end{equation}
Hence by (\ref{id2.13}), (\ref{id3.16}) and (\ref{id3.17}),
(\ref{id3.13}) gives
\begin{align}\label{id3.18}
&\frac{1}{p+1}\int_{\mathbb{R}^N}w_{x_0}^{p+1}dy\sum\limits_{j=1}^N\partial_{jk}m(x_0)\big(h^{-1}c_h^j\big)\nonumber\\
=&\frac{1}{2\lambda}\int_{\mathbb{R}^N}\Big[\sum\limits_{i,j=1}^N\partial_{ij}m(x_0)x_{1,i}y_j
+\frac{1}{2}\sum\limits_{i,j,l=1}^N\partial_{ijl}m(x_0)y_iy_jy_l\Big]w_{x_h}^p\partial_k
w_{x_h}dy+\mathrm{o}(1).
\end{align}
Using integration by parts, we obtain
\begin{align*}
\begin{cases}
&\int\limits_{\mathbb{R}^N}y_jw_{x_h}^p\partial_k w_{x_h}dy
=-\frac{\delta_{jk}}{p+1}\int\limits_{\mathbb{R}^N}w_{x_h}^{p+1}dy\,,\\
&\int\limits_{\mathbb{R}^N}y_iy_jy_lw_{x_h}^p\partial_k w_{x_h}dy
=-\frac{\delta_{ik}\delta_{jl}+\delta_{jk}\delta_{il}+\delta_{lk}\delta_{ij}}{N(p+1)}
\int\limits_{\mathbb{R}^N}|y|^2w_{x_h}^{p+1}dy\,,
\end{cases}
\end{align*}
where $\delta$ is the Kronecker symbol. Hence by (\ref{id3.18}),
$|c_h^j|=\mathrm{O}(h)$ for $j=1,\cdots,N$. Moreover,
by~(\ref{id3.2}), we obtain~(\ref{id3.9}) and complete the proof.
\end{proof}

Let us now compute $d\,''(\lambda)$. From~(\ref{id1.5}), it is easy
to get
\begin{equation*}
d\,'(\lambda)=\frac{1}{2}\int\limits_{\mathbb{R}^N}u_h^2dx
\end{equation*}
and hence
\begin{equation}\label{id3.6}
d\,''(\lambda)=\int\limits_{\mathbb{R}^N}u_h\frac{\partial
u_h}{\partial\lambda}dx=h^N\int\limits_{\mathbb{R}^N}v_hR_hdy\,.
\end{equation}

Using integration by parts and~(\ref{eq3.4}), we have
\begin{eqnarray}\label{eq3.7}
\int\limits_{\mathbb{R}^N}v_hR_0dy=\int\limits_{\mathbb{R}^N}v_h
\lambda^{-1}\big(\frac{1}{p-1}v_h+\frac{1}{2}y\cdot\nabla v_h\big)dy
=\lambda^{-1}\big(\frac{1}{p-1}-\frac{N}{4}\big)\int\limits_{\mathbb{R}^N}v_h^2dy=0\,,
\end{eqnarray}
since $p=1+\frac{4}{N}$. Hence, by~(\ref{id3.6}) and
Theorem~\ref{thm3.2}, we have
\begin{align*}
\frac{d\,''(\lambda)}{h^N}
=&\int\limits_{\mathbb{R}^N}v_h\Big[R_0+\sum\limits_{j=1}^Nc_h^j\partial_j w_{x_h}+h^2R_1+R_h^\perp\Big]dy\\
=&\int\limits_{\mathbb{R}^N}v_h\Big[\sum\limits_{j=1}^Nc_h^j\partial_j
w_{x_h}+h^2R_1+R_h^\perp\Big]dy
\quad\quad\big({\rm{because}\,\,}\int\limits_{\mathbb{R}^N}v_hR_0dy=0\big)\\
=&\int\limits_{\mathbb{R}^N}R_h\Big[\sum\limits_{j=1}^Nc_h^jL_h\partial_j
w_{x_h}
+h^2L_hR_1+L_hR_h^\perp\Big]dy\quad\quad\big({\rm{because}}\quad L_hR_h=v_h\big)\\
=&\int\limits_{\mathbb{R}^N}\Big[R_0+\sum\limits_{j=1}^Nc_h^j\partial_j
w_{x_h}+h^2R_1+R_h^\perp\Big]
\Big[\sum\limits_{j=1}^Nc_h^jL_h\partial_j
w_{x_h}+h^2L_hR_1+L_hR_h^\perp\Big]dy\,.
\end{align*}
Therefore, by~(\ref{id2.10}), (\ref{id3.10}) and
$c_h^j=\mathrm{O}(h)$,
\begin{align}\label{id3.20}
\frac{d\,''(\lambda)}{h^N}
=&\int\limits_{\mathbb{R}^N}R_0\Big[v_h-L_hR_0\Big]dy
+\sum\limits_{j,k=1}^Nc_h^jc_h^k\int\limits_{\mathbb{R}^N}\partial_k
w_{x_h}\big(L_h\partial_j w_{x_h}\big)dy\nonumber\\
&+h^4\int\limits_{\mathbb{R}^N}R_1\big(L_hR_1\big)dy+\mathrm{O}(h^5)\,.
\end{align}

For the integral
$\int\limits_{\mathbb{R}^N}R_0\Big[v_h-L_hR_0\Big]dy$, by
(\ref{eq3.8})and using integration by parts, we have
\begin{align*}
\int\limits_{\mathbb{R}^N}R_0\Big[v_h-L_hR_0\Big]dy
=&\int\limits_{\mathbb{R}^N}\lambda^{-1}\big(\frac{1}{p-1}v_h+\frac{1}{2}y\cdot\nabla
v_h\big)\Big[\frac{h}{2\lambda}y\cdot\nabla
m(hy+x_h)v_h^p\Big]dy\\
=&\frac{1}{2\lambda^2}\int\limits_{\mathbb{R}^N}\frac{N}{4(N+2)}\Big[hy\cdot\nabla
m(hy+x_h)-h^2\sum\limits_{i,j=1}^N\partial_{ij}m(hy+x_h)y_iy_j\Big]v_h^{p+1}dy.
\end{align*}
Note that by Lemma~\ref{lem2.4}, \ref{lem3.1} and
Theorem~\ref{thm3.2}, we have
\begin{align*}
&hy\cdot\nabla m(hy+x_h)-h^2\sum\limits_{i,j=1}^N\partial_{ij}m(hy+x_h)y_iy_j\\
=&hy\cdot\nabla
m(x_h)-\frac{h^3}{2}\sum\limits_{i,j,k=1}^N\partial_{ijk}m(x_h)y_iy_jy_k
-\frac{h^4}{3}\sum\limits_{i,j,k,l=1}^N\partial_{ijkl}m(x_h)y_iy_jy_ky_l+\mathrm{o}(h^4)\,,
\end{align*}
and
\begin{eqnarray}\label{id3.19}
v_h^p=w_{x_h}^p+h^2pw_{x_h}^{p-1}\phi_2+\mathrm{O}(h^3)\,.
\end{eqnarray}
Hence
\begin{align}\label{id3.5}
\int\limits_{\mathbb{R}^N}R_0\Big[v_h-L_hR_0\Big]dy
=&\frac{N}{8(N+2)}\lambda^{-2}\int\limits_{\mathbb{R}^N}
\Big[-\frac{h^4}{3}\sum\limits_{i,j,k,l=1}^N
\partial_{ijkl}m(x_h)y_iy_jy_ky_l\Big]w_{x_h}^{p+1}dy+\mathrm{o}(h^4)\nonumber\\
=&-\frac{h^4}{8(N+2)^2}\lambda^{-2}\int\limits_{\mathbb{R}^N}
|y|^4w_{x_h}^{p+1}dy\Delta^2m(x_0)+\mathrm{o}(h^4)\nonumber\\
=&-\frac{h^4}{8(N+2)^2}\lambda^{-3}m(x_0)^{-\frac{N}{2}-1}\int\limits_{\mathbb{R}^N}
|y|^4w^{p+1}dy\Delta^2m(x_0)+\mathrm{o}(h^4)\,.
\end{align}
Here we have used the following identities:
\begin{align*}
\begin{cases}
&\int\limits_{\mathbb{R}^N}y_iw_{x_h}^{p+1}dy=
\int\limits_{\mathbb{R}^N}y_iy_jy_kw_{x_h}^{p+1}dy=0\,,\quad\hbox{
for\,\,all
}\:i,j,k=1,\cdots,N\,;\\
&\int\limits_{\mathbb{R}^N}y_iy_jy_ky_lw_{x_h}^{p+1}dy=0\,,\quad
{\rm{if}}\
y_iy_jy_ky_l\ {\rm{is \ an\ odd\ function\ on\ one\ of\ its\ variate}}\,;\\
&\int\limits_{\mathbb{R}^N}y_i^4w_{x_h}^{p+1}dy=\frac{3}{N(N+2)}
\int\limits_{\mathbb{R}^N}|y|^4w_{x_h}^{p+1}dy\,,
\quad {\rm{for\,\,all}}\ i=1,\cdots,N\,; \\
&\int\limits_{\mathbb{R}^N}y_i^2y_j^2w_{x_h}^{p+1}dy=\frac{1}{N(N+2)}
\int\limits_{\mathbb{R}^N}|y|^4w_{x_h}^{p+1}dy\,,
\quad{\rm{for\,\,all}}\
 i\neq j\,,
 \end{cases}
\end{align*}
which can be proved by polar coordinates.

For the sum
$\sum\limits_{j,k=1}^Nc_h^jc_h^k\int\limits_{\mathbb{R}^N}\partial_k
w_{x_h}\big(L_h\partial_j w_{x_h}\big)dy$, we may use
(\ref{id2.13}) and (\ref{id3.9}) to get
\begin{align}\label{id3.8}
&\sum\limits_{j,k=1}^Nc_h^jc_h^k\int_{\mathbb{R}^N}\partial_k
w_{x_h}\left(L_h\partial_j
w_{x_h}\right)dy\nonumber\\
=&\frac{h^4}{p+1}\int_{\mathbb{R}^N}w_{x_h}^{p+1}dy\sum\limits_{j,k=1}^N(h^{-1}c_h^j)(h^{-1}c_h^k)\partial
_{jk}m(x_0)+\mathrm{o}(h^4)\nonumber\\
=&\frac{h^4}{8N(N+2)}\lambda^{-3}
m(x_0)^{-\frac{N}{2}-1}\frac{\big(\int\limits_{\mathbb{R}^N}|y|^2w^{p+1}dy\big)^2}
{\int\limits_{\mathbb{R}^N}w^{p+1}dy}\nabla(\Delta m)(x_0)
\cdot\big[\nabla^2m(x_0)\big]^{-1}\nabla(\Delta m)(x_0)+\mathrm{o}(h^4).\nonumber\\
\end{align}

For the integral
$h^4\int\limits_{\mathbb{R}^N}R_1\big(L_hR_1\big)dy$, by
(\ref{eq3.2}), it is obvious that $R_1(\lambda^{-\frac{1}{2}}y)$
satisfies
\begin{equation}
\Delta
R-R+pw^{p-1}R-\frac{1}{2}\lambda^{\frac{N}{4}-2}m(x_h)^{-\frac{N}{4}-1}
\sum\limits_{i,j=1}^N\partial_{ij}m(x_0)y_iy_jw^p=0\,.
\end{equation}
Hence
\begin{align}\label{id3.111}
&h^4\int\limits_{\mathbb{R}^N}R_1\big(L_hR_1\big)dy
=h^4\int\limits_{\mathbb{R}^N}R_1\big(L_{x_h}R_1\big)dy+\mathrm{O}(h^6)\nonumber\\
=&\frac{h^4}{4}\lambda^{-3}m(x_0)^{-\frac{N}{2}-2}\sum\limits_{i,j,k,l=1}^N
\partial_{ij}m(x_0)\partial_{kl}m(x_0)\int\limits_{\mathbb{R}^N}y_iy_iw^p
L_0^{-1}\big(y_ky_lw^p\big)dy+\mathrm{O}(h^6)\nonumber\\
=&\frac{h^4}{4N^2}\lambda^{-3}m(x_0)^{-\frac{N}{2}-2}|\Delta
m(x_0)|^2\int\limits_{\mathbb{R}^N}r^2w^pL_0^{-1}(r^2w^p)dy\nonumber\\
&+\frac{h^4}{2N(N+2)}\lambda^{-3}m(x_0)^{-\frac{N}{2}-2}\|\nabla^2m(x_0)\|_2^2
\int\limits_{\mathbb{R}^N}r^2w^p\Phi_0(r)dy\nonumber\\
&-\frac{h^4}{2N^2(N+2)}\lambda^{-3}m(x_0)^{-\frac{N}{2}-2}|\Delta
m(x_0)|^2\int\limits_{\mathbb{R}^N}r^2w^p\Phi_0(r)dy+\mathrm{O}(h^6).
\end{align}
Here $\|\nabla^2m(x_0)\|_2^2=\sum\limits_{i,j=1}^Nm_{ij}^2(x_0)$ and
we have used the following identities:
\begin{align}
&\int\limits_{\mathbb{R}^N}
y^2_{N}w^pL_0^{-1}(y^2_Nw^p)dy=\frac{1}{N^2}\int\limits_{\mathbb{R}^N}r^2w^pL_0^{-1}(r^2w^p)dy
+\frac{2(N-1)}{N^2(N+2)}\int\limits_{\mathbb{R}^N}
r^2w^p\Phi_0(r)dy\,,\label{id3.21-1}\\
&\int\limits_{\mathbb{R}^N}
y^2_{N-1}w^pL_0^{-1}(y^2_Nw^p)dy=\frac{1}{N^2}\int\limits_{\mathbb{R}^N}r^2w^pL_0^{-1}(r^2w^p)dy
-\frac{2}{N^2(N+2)}\int\limits_{\mathbb{R}^N}r^2w^p\Phi_0(r)dy\,,\label{id3.21-2}\\
&\int\limits_{\mathbb{R}^N}
y_{N-1}y_Nw^pL_0^{-1}(y_{N-1}y_Nw^p)dy=\frac{1}{N(N+2)}
\int\limits_{\mathbb{R}^N}r^2w^p\Phi_0(r)dy\,,\label{id3.21-3}
\end{align}
where $\Phi_0$ satisfies (\ref{id1.7}), which we have proved in
Appendix B.

Therefore, combining (\ref{id3.20}), (\ref{id3.5}), (\ref{id3.8})
and (\ref{id3.111}), we obtain
\begin{align*}
&\frac{d\,''(\lambda)}{h^N}+\mathrm{o}(h^4)\\
=&-\frac{h^4}{8(N+2)^2}\lambda^{-3}m(x_0)^{-\frac{N}{2}-1}
\int\limits_{\mathbb{R}^N}|y|^4w^{p+1}dy\Delta^2m(x_0)\\
&+\frac{h^4}{8N(N+2)}\lambda^{-3}m(x_0)^{-\frac{N}{2}-1}
\frac{\big(\int\limits_{\mathbb{R}^N}|y|^2w^{p+1}dy\big)^2}
{\int\limits_{\mathbb{R}^N}w^{p+1}dy}\nabla(\Delta m)(x_0)
\cdot\big[\nabla^2m(x_0)\big]^{-1}\nabla(\Delta
m)(x_0)\\
&+\frac{h^4}{4N^2}\lambda^{-3}m(x_0)^{-\frac{N}{2}-2}|\Delta
m(x_0)|^2
\int\limits_{\mathbb{R}^N}|y|^2w^pL_0^{-1}\big(|y|^2w^p\big)dy\\
&+\frac{h^4}{2N^2(N+2)}\lambda^{-3}m(x_0)^{-\frac{N}{2}-2}\Big[N\|\nabla^2m(x_0)\|_2^2-|\Delta
m(x_0)|^2\Big] \int\limits_{\mathbb{R}^N}|y|^2w^p\Phi_0(|y|)dy.
\end{align*}
Consequently,
\begin{align*}
\frac{8(N+2)^2m(x_0)^{\frac{N}{2}+2}\lambda^3}{h^{N+4}
\int\limits_{\mathbb{R}^N}|y|^4w^{p+1}dy}d\,''(\lambda)
=&C_{N,1}|\Delta
m(x_0)|^2+C_{N,2}\big(N\|\nabla^2m(x_0)\|_2^2-|\Delta
m(x_0)|^2\big)\\
&+C_{N,3}m(x_0)\Big[\nabla(\Delta
m)(x_0)\cdot\big[\nabla^2m(x_0)\big]^{-1}\nabla(\Delta
m)(x_0)\Big]\\
&-m(x_0)\Delta^2 m(x_0)+\mathrm{o}(1)\,,
\end{align*}
where $C_{N,1},C_{N,2},C_{N,3}$ are constants given by
(\ref{id1.9-1}), (\ref{id1.9-2}), (\ref{id1.9-3}), respectively.

Now we may prove Theorem~\ref{thm1.1} as follows: Suppose that
$x_0$ is a non-degenerate local maximum point of the function
$m(x)$, then the Hessian matrix $\nabla^2m(x_0)$ of m at $x_0$ is
negative definite. By Theorem~\ref{thm2.7}, we have $n(L_h)=1$. On
the other hand, we have $p(d\,'')=1$. Thus $\psi_h$ is orbital
stable by the orbital stability criteria of
\cite{[GS1]}-\cite{[GS2]}. For orbital instability, we denote the
number of positive eigenvalues of the Hessian matrix
$\nabla^2m(x_0)$ by $n$. Then by Theorem~\ref{thm2.7}, we obtain
$n(L_h)=n+1$. On the other hand, we have $p(d\,'')=1$. Thus by the
instability criteria of \cite{[GS2]}, we conclude that $\psi_h$ is
orbitally unstable if $n$ is odd. This may complete the proof of
Theorem~\ref{thm1.1}.

\section{Proof of Theorem \ref{thm1.2}-\ref{thm1.4}}

\ \ \ \ In this section, we may generalize the argument of Section
2 and 3 to prove Theorem~\ref{thm1.2}-\ref{thm1.4}. Let
$v_h(y):=u_h(hy+x_h)$, where $u_h$ is a single-spike bound state
of (\ref{id1.3}) with a unique local maximum point at~$x_h$. Then
$v_h$ satisfies
\begin{equation}\label{eq4.1}
\Delta\,v_h-\Big[V(hy+x_h)+\lambda\Big]v_h
+m(hy+x_h)v_h^p=0\quad\hbox{ in }\:\mathbb{R}^N\,.
\end{equation}

Suppose~(\ref{id2.2}) hold. By~(\ref{id2.3}) and~\cite{[WXZ]}, we
have
\begin{equation}\label{eq4.2}
m(x_0)\nabla V(x_0)=\frac{N}{2}\left[V(x_0)+\lambda\right]\nabla
m(x_0)\,,
\end{equation}
so $x_0$ may depend on $\lambda$. Note that by (\ref{eq4.2}),
$\nabla m(x_0)=0$ if and only if $\nabla V(x_0)=0.$ By direct
computation on the function $G$,
\begin{align*}
\partial_{ij}G(x_0)=m(x_0)^{-\frac{N}{2}-1}&\Big[m(x_0)\partial_{ij}V(x_0)
+(1-\frac{N}{2})\partial_iV(x_0)\partial_jm(x_0)\\
&-\frac{N}{2}\left[V(x_0) +\lambda\right]\partial_{ij}m(x_0)\Big]\,.
\end{align*}
In particular, if $\nabla m(x_0)=0$, then
\begin{equation*}
\nabla^2G(x_0)=m(x_0)^{-\frac{N}{2}-1}\Big[m(x_0)\nabla^2V(x_0)-\frac{N}{2}\big[V(x_0)
+\lambda\big]\nabla^2m(x_0)\Big]\,.
\end{equation*}

Using the identity~(\ref{id2.3}), one may follow the arguments of
Lemma~\ref{lem2.2}-\ref{lem2.4} to get the uniqueness of $u_h$ and
\begin{align}
x_h=&\,x_0+\mathrm{o}(h)\label{id4.2}\,;\\
v_h=&\,w_{x_h}+h\phi_1+h^2\phi_2+\mathrm{o}(h^2)\label{id4.3}\,,
\end{align}
where $\phi_1$ and $\phi_2$ satisfy
$\nabla\phi_1(0)=\nabla\phi_2(0)=0\,,$
\begin{equation}\label{eq4.6}
\Delta\phi_1-\left[V(x_0)+\lambda\right]\phi_1+m(x_0)pw_{x_0}^{p-1}\phi_1
-y\cdot\nabla V(x_0)w_{x_0}+y\cdot\nabla m(x_0)w_{x_0}^p=0\,,
\end{equation}
and
\begin{align}\label{eq4.7}
&\Delta\phi_2-\big[V(x_h)+\lambda\big]\phi_2+m(x_h)pw_{x_h}^{p-1}\phi_2
-y\cdot\nabla
V(x_0)\phi_1-\frac{1}{2}\sum\limits_{i,j=1}^N\partial_{ij}V(x_0)y_iy_jw_{x_h}\nonumber\\
&+y\cdot\nabla m(x_0)pw_{x_h}^{p-1}\phi_1
+\frac{1}{2}\sum\limits_{i,j=1}^N\partial_{ij}m(x_0)y_iy_jw_{x_h}^p
+\frac{1}{2}m(x_0)p(p-1)w_{x_h}^{p-2}\phi_1^2=0\,.
\end{align}
Here we have used the hypothesis that $x_0$ is a non-degenerate
point of the function~$G$. And the only difference in the proof is
that we need to estimate the term
\begin{equation*}
\frac{1}{p+1}\nabla m(x_0)\int\limits_{\mathbb{R}^N}v_h^{p+1}dy
-\frac{1}{2}\nabla V(x_0)\int\limits_{\mathbb{R}^N}v_h^2dy\,,
\end{equation*}
to estimate which one may use the following Pohozaev identity
(cf.~\cite{[P]})
\begin{align*}
&\int\limits_{\mathbb{R}^N}\Big[\frac{2}{N+2}m(hy+x_h)+\frac{h}{p+1}y\cdot\nabla
m(hy+x_h)\Big]v_h^{p+1}dy\nonumber\\
=&\int\limits_{\mathbb{R}^N}\Big[V(hy+x_h)+\lambda+\frac{h}{2}y\cdot\nabla
V(hy+x_h)\Big]v_h^2dy\,.
\end{align*}

For the small eigenvalue estimates of $L_h$, one may generalize
the idea of Theorem~\ref{thm2.7} to get
\begin{thm}\label{thm4.1}
For $h$ small enough, the eigenvalue problem
\begin{align}
L_h\varphi_h =\mu_h\varphi_h
\end{align}
has exactly $N$ eigenvalues $\mu_h^j\,,  j=1,\cdots N$, in the interval $[\frac{1}{2}\mu_1,\frac{1}{2}\mu_{N+2}]$, which satisfy
and \begin{align}\label{id2.14} \frac{\mu_h^j}{h^2}\rightarrow
c_0\nu_{j}\,,\quad\hbox{ as }\:h\to 0\,,\quad\hbox{ for }\:
j=1,\cdots N\,,
\end{align} where $\mu_1$ and $\mu_{N+2}$ are defined
Lemma~\ref{lem2.6}, $\nu_j$'s are the eigenvalues of the Hessian
matrix $\nabla^2 G(x_0)$, and
$c_0=-\frac{m(x_0)^{N/2}}{V(x_0)+\lambda}=-G(x_0)^{-1}$ is a negative
constant. Furthermore, the corresponding eigenfunctions
$\varphi_h^j$'s satisfy
\begin{equation}\label{410}
\varphi_h^j=\sum\limits_{i=1}^N\big[a_{ij}+\mathrm{o}(1)\big]\big(\partial_{i}
w_{x_h}+h\psi_i\big)+\mathrm{O}(h^2)\,,\quad j=1,\cdots,N\,,
\end{equation}
where each $\psi_i$ is the solution of
\begin{align}\label{eq4.3}
&\Delta\psi_i-\Big[V(x_h)+\lambda\Big]\psi_i+m(x_h)pw_{x_h}^{p-1}\psi_i\nonumber\\
&+\Big[-y\cdot\nabla V(x_h)+y\cdot\nabla m(x_h)pw_{x_h}^{p-1}
+m(x_h)p(p-1)w_{x_h}^{p-2}\phi_1\Big]\partial_iw_{x_h}=0\,,
\end{align}
and $\textbf{a}_j=(a_{1j},\cdots,a_{Nj})^T$ is the eigenvector
corresponding to $\nu_j$, namely,
\begin{equation}
\nabla^2 G(x_0)\textbf{a}_j=\nu_j\textbf{a}_j\,.
\end{equation}
\end{thm}
\noindent {\bf Remark 5:} (1) To prove it, one may follow the arguments in
the proof of Theorem~\ref{thm2.7} and use the following identity
\begin{equation}\label{eq4.4}
\int\limits_{\mathbb{R}^N}\partial_{k} w_{x_h}L_h\big(\partial_{j}
w_{x_h}+h\psi_j\big)dy=-\frac{h^2}{N+2}\int\limits_{\mathbb{R}^N}w^{p+1}dy
\partial_{jk}G(x_0)+\mathrm{o}(h^2),
\end{equation}
to replace (\ref{id2.13}) (see Appendix C). The main difference
between Theorem~\ref{thm2.7} and~\ref{thm4.1} is that (\ref{410})
has the solution $\psi_i$ of (\ref{eq4.3}) which comes from
\begin{align}\label{id4.7}
L_h\partial_iw_{x_h}=h\Big[&-y\cdot\nabla V(x_0)+y\cdot\nabla
m(x_0)pw_{x_h}^{p-1}\nonumber\\
&+m(x_0)p(p-1)w_{x_h}^{p-2}\phi_1\Big]\partial_iw_{x_h}+\mathrm{O}(h^2).
\end{align}

(2) Let $n$ be the number of negative eigenvalues of the matrix $\delta^2G(x_0)$, then similar to the Remark 4(3), the number of positive eigenvalues of $L_h$ equals $n+1$, i.e., $n(L_h)=n+1$.

Since the potential function $V$ is nonzero, then $x_0$ may depend
on $\lambda$ and the asymptotic expansion of $d\,''(\lambda)$
becomes more complicated. Indeed, when $m\equiv1$ and $\Delta
V(x_0)\neq0$, the result in~\cite{[lw5]} shows that the effect of
potential function $V$ on $d\,''(\lambda)$ is $\mathrm{O}(h^2)$. On
the other hand, when $V\equiv0$ and condition~(\ref{con1.1}) holds,
the effect of $m$ on $d\,''(\lambda)$ is $\mathrm{O}(h^4)$ (see
Section~3). Generally, when both $m$ and $V$ are not constant, we
may show
\begin{enumerate}[(I)]
\item The effect of $V$ and $m$ on $d\,''(\lambda)$ is $\mathrm{O}(1)$
if $\nabla V(x_0)\neq0$ (see Theorem~\ref{thm1.2});

\item The effect of $V$ and $m$ on $d\,''(\lambda)$ is
$\mathrm{O}(h^2)$ if $\nabla V(x_0)=0$ and $\Delta V(x_0)\neq0$ (see
Theorem~\ref{thm1.3});

\item The effect of $V$ and $m$ on $d\,''(\lambda)$ is
$\mathrm{O}(h^4)$ if $\nabla V(x_0)=0\,,\Delta V(x_0)=0$ and some
local condition hold (see Theorem~\ref{thm1.4}).
\end{enumerate}

\noindent Now we divide three cases to prove these results.\\\\
\textbf{Case I: $\nabla V(x_0)\neq0$.}\\

Let $R_h:=\frac{\partial u_h}{\partial\lambda}(hy+x_h)$.
Then~(\ref{eq3.4}) and (\ref{eq3.7}) hold. Hence one may apply the
idea of Theorem~\ref{thm3.2} to get
\begin{align}\label{id4.4}
R_h=\sum\limits_{i=1}^Nc_h^i\big(\partial_{i}
w_{x_h}+h\psi_i\big)+R_0+R_h^\perp\,,
\end{align}
where as $h\rightarrow0$, $\textbf{c}_h=(c_h^1,\cdots,c_h^N)$
satisfies
\begin{equation}\label{id4.5}
\nabla^2G(x_0)(h\textbf{c}_h)\rightarrow-\frac{N}{2}m(x_0)^{-\frac{N}{2}-1}\nabla
m(x_0)\,,
\end{equation}
and
\begin{equation}\label{id4.6}
R_0=\big[V(x_h)+\lambda\big]^{-1}\big(\frac{1}{p-1}v_h+\frac{1}{2}y\cdot\nabla
v_h\big)\,,\,\,\,R_h^\perp=\mathrm{O}(h)\,.
\end{equation}
Thus
\begin{align*}
\frac{d\,''(\lambda)}{h^N}=&\int\limits_{\mathbb{R}^N}v_hR_hdy
=\int\limits_{\mathbb{R}^N}v_h\Big[\sum\limits_{i=1}^Nc_h^i\big(\partial_{i}
w_{x_h}+h\psi_i\big)+R_0+R_h^\perp\Big]dy\\
=&\int\limits_{\mathbb{R}^N}v_h\sum\limits_{i=1}^Nc_h^i\big(\partial_{i}
w_{x_h}+h\psi_i\big)dy+\mathrm{O}(h)\quad\quad\Big({\rm{because}}\,
\int\limits_{\mathbb{R}^N}v_hR_0dy=0\Big)\\
=&\int\limits_{\mathbb{R}^N}R_h\sum\limits_{i=1}^Nc_h^iL_h\big(\partial_{i}
w_{x_h}+h\psi_i\big)dy+\mathrm{O}(h)\quad\quad\Big({\rm{because}}\,
L_hR_h=v_h\Big)\\
=&\int\limits_{\mathbb{R}^N}\Big[\sum\limits_{k=1}^Nc_h^k\big(\partial_{k}
w_{x_h}+h\psi_k\big)+R_0+R_h^\perp\Big]\sum\limits_{i=1}^Nc_h^iL_h\big(\partial_{i}
w_{x_h}+h\psi_i\big)dy+\mathrm{O}(h)\,.
\end{align*}
Therefore, by (\ref{eq4.3}), (\ref{id4.7}), (\ref{eq4.4}),
(\ref{id4.5}) and (\ref{id4.6}), we obtain
\begin{equation}\label{id4.1}
\frac{d\,''(\lambda)}{h^N}=-\frac{N^2}{4(N+2)}m(x_0)^{-N-2}\int\limits_{\mathbb{R}^N}w^{p+1}dy
\nabla m(x_0)\cdot\big[\nabla^2G(x_0)\big]^{-1}\nabla
m(x_0)+\mathrm{O}(h)\,.
\end{equation}
Consequently, if $x_0$ is a non-degenerate local minimum point of
$G$, then the Hessian matrix $\nabla^2G(x_0)$ is positive
definite. By Theorem~\ref{thm4.1}, we have $n(L_h)=1$. On the
other hand, by~(\ref{id4.1}), we have $p(d\,'')=0$. Thus we
complete the proof of Theorem~{\ref{thm1.2}} by the orbital
instability criteria
of~\cite{[GS1]}-\cite{[GS2]}.\\\\
\textbf{Case II: $\nabla V(x_0)=0$ and $\Delta V(x_0)\neq0$.}\\

Firstly, note that in this case, $\phi_1\equiv0$ and
$\psi_i\equiv0$. Then one may apply the idea of Lemma~\ref{lem3.1}
and Theorem~\ref{thm3.2} to obtain
\begin{align}
x_h=&x_0+h^2\textbf{x}_1+\mathrm{O}(h^3)\,;\label{id4.10}\\
R_h=&R_0+\sum\limits_{j=1}^Nc_h^j\partial_j
w_{x_h}+h^2R_1+R_h^\perp\label{id4.8}\,,
\end{align}
where $\textbf{x}_1\in\mathbb{R}^N$ satisfies
\begin{align}\label{id4.15}
\nabla^2G(x_0)\textbf{x}_1=&
-\frac{N+2}{4N}\big[V(x_0)+\lambda\big]^{-1}m(x_0)^{-\frac{N}{2}}
\left(\frac{\int\limits_{\mathbb{R}^N}|y|^2w^2dy}
{\int\limits_{\mathbb{R}^N}w^{p+1}dy}\right)\nabla(\Delta V)(x_0)\nonumber\\
&+\frac{1}{4}m(x_0)^{-\frac{N}{2}-1}\left(\frac{\int\limits_{\mathbb{R}^N}|y|^2w^{p+1}dy}
{\int\limits_{\mathbb{R}^N}w^{p+1}dy}\right) \nabla(\Delta
m)(x_0)\,,
\end{align}
$R_1$ satisfies
\begin{align}\label{eq4.5}
&\Delta R_1-\big[V(x_h)+\lambda\big]
R_1+m(x_h)pw_{x_h}^{p-1}R_1\nonumber\\
&+\big[V(x_h)+\lambda\big]^{-1}
\Big[\sum\limits_{i,j=1}^N\partial_{ij}V(x_0)y_iy_jw_{x_h}-\frac{1}{2}
\sum\limits_{i,j=1}^N\partial_{ij}m(x_0)y_iy_jw_{x_h}^p\Big]=0\,,
\end{align}
$R_h^\perp=\mathrm{O}(h^3)$ and $c_h^j=\mathrm{O}(h)$ for
$j=1,\cdots,N$. Moreover, $\textbf{c}_h:=(c_h^1,\cdots,c_h^N)$
satisfies
\begin{equation}\label{id4.17}
\nabla^2G(x_0)\big(h^{-1}\textbf{c}_h\big)=\textbf{c}_0+\mathrm{o}(1)\,,
\end{equation}
where
\begin{align}\label{id4.14}
\textbf{c}_0
=&-\big[V(x_0)+\lambda\big]^{-1}m(x_0)^{-\frac{N}{2}}\nabla^2V(x_0)
\textbf{x}_1\nonumber\\
&-\frac{N+2}{2N}\big[V(x_0)+\lambda\big]^{-2}m(x_0)^{-\frac{N}{2}}
\left(\frac{\int\limits_{\mathbb{R}^N}|y|^2w^2dy}
{\int\limits_{\mathbb{R}^N}w^{p+1}dy}\right)\nabla(\Delta V)(x_0)\nonumber\\
&+\frac{1}{4}\big[V(x_0)+\lambda\big]^{-1}m(x_0)^{-\frac{N}{2}-1}
\left(\frac{\int\limits_{\mathbb{R}^N}|y|^2w^{p+1}dy}
{\int\limits_{\mathbb{R}^N}w^{p+1}dy}\right)\nabla(\Delta m)(x_0).
\end{align}
Hence
\begin{align*}
\frac{d\,''(\lambda)}{h^N}=&\int\limits_{\mathbb{R}^N}v_hR_hdy
=\int\limits_{\mathbb{R}^N}v_h\Big[R_0+\sum\limits_{j=1}^Nc_h^j\partial_j w_{x_h}+h^2R_1+R_h^\perp\Big]dy\\
=&\int\limits_{\mathbb{R}^N}v_h\Big[\sum\limits_{j=1}^Nc_h^j\partial_j
w_{x_h}+h^2R_1+R_h^\perp\Big]dy
\quad\Big({\rm{because}}\,\int\limits_{\mathbb{R}^N}v_hR_0dy=0\Big)\\
=&\int\limits_{\mathbb{R}^N}R_h\Big[\sum\limits_{j=1}^Nc_h^jL_h\partial_j
w_{x_h}
+h^2L_hR_1+L_hR_h^\perp\Big]dy\quad\Big({\rm{because}}\,L_hR_h=v_h\Big)\\
=&\int\limits_{\mathbb{R}^N}\Big[R_0+\sum\limits_{k=1}^Nc_h^k\partial_k
w_{x_h}+h^2R_1+R_h^\perp\Big]
\Big[\sum\limits_{j=1}^Nc_h^jL_h\partial_j
w_{x_h}+h^2L_hR_1+L_hR_h^\perp\Big]dy\,.
\end{align*}
Therefore, by (\ref{eq4.3}), (\ref{id4.7}) and (\ref{id4.8}), we
obtain
\begin{align}\label{id4.9}
\frac{d\,''(\lambda)}{h^N}
=&\int\limits_{\mathbb{R}^N}R_0\big[v_h-L_hR_0\big]dy
+\sum\limits_{j,k=1}^Nc_h^jc_h^k\int\limits_{\mathbb{R}^N}\partial_k
w_{x_h}L_h\big(\partial_j w_{x_h}\big)dy\nonumber\\
&+h^4\int\limits_{\mathbb{R}^N}R_1\big(L_hR_1\big)dy+\mathrm{O}(h^5)\,.
\end{align}

For the integral
$\int\limits_{\mathbb{R}^N}R_0\big[v_h-L_hR_0\big]dy$, by direct
computation, we have
\begin{align}\label{id4.16}
v_h-L_hR_0=&-\big[V(x_h)+\lambda\big]^{-1}\Big[V(hy+x_h)-V(x_h)+\frac{h}{2}y\cdot\nabla
V(hy+x_h)\Big]v_h\nonumber\\
&+\frac{h}{2}\big[V(x_h)+\lambda\big]^{-1}y\cdot\nabla m(hy+x_h)v_h^p.
\end{align}
Thus by (\ref{id4.3}), (\ref{id4.10}) and (\ref{eq2.2-1}), we obtain
\begin{equation}\label{id4.12}
\int\limits_{\mathbb{R}^N}R_0\big[v_h-L_hR_0\big]dy
=\frac{h^2}{2N}\big[V(x_0)+\lambda\big]^{-3}m(x_0)^{-\frac{N}{2}}\int\limits_{\mathbb{R}^N}|y|^2w^2dy
\Delta V(x_0)+\mathrm{O}(h^4)\,.
\end{equation}

For the sum
$\sum\limits_{j,k=1}^Nc_h^jc_h^k\int\limits_{\mathbb{R}^N}\partial_k
w_{x_h}\big(L_h\partial_j w_{x_h}\big)dy$, by (\ref{eq4.3}),
(\ref{id4.7}) and $c_h^j=\mathrm{O}(h)$ for $j=1,\cdots,N$, we have
\begin{equation}\label{id4.11}
\sum\limits_{j,k=1}^Nc_h^jc_h^k\int\limits_{\mathbb{R}^N}\partial_k
w_{x_h}\big(L_h\partial_j w_{x_h}\big)dy=\mathrm{O}(h^4)\,.
\end{equation}
Combining (\ref{id4.12}), (\ref{id4.11}) and (\ref{id4.9}), we
obtain
\begin{equation}\label{id4.13}
\frac{d\,''(\lambda)}{h^N}=\frac{h^2}{2N}\big[V(x_0)+\lambda\big]^{-3}
m(x_0)^{-\frac{N}{2}}\int\limits_{\mathbb{R}^N}|y|^2w^2dy \Delta
V(x_0)+\mathrm{O}(h^4)\,.
\end{equation}
Consequently, by~(\ref{id4.13}), we have
$p(d\,'')=\frac{1}{2}(1+\frac{\Delta V(x_0)}{|\Delta V(x_0)|})$. On
the other hand, by Theorem~\ref{thm4.1}, we have $n(L_h)=n+1$. Thus
we complete the proof of Theorem~{\ref{thm1.3}} by the orbital
stability and instability criteria of~\cite{[GS1]}-\cite{[GS2]}.\\\\
\textbf{Case III: $\nabla V(x_0)=0\,,\Delta V(x_0)=0$.}\\

In this case, we shall use (\ref{id4.14}), (\ref{id4.15}) and
(\ref{id4.9}) to compute the $\mathrm{O}(h^4)$ term of
$d\,''(\lambda)/h^N$.

For the integral
$\int\limits_{\mathbb{R}^N}R_0\Big[v_h-L_hR_0\Big]dy$, by
(\ref{id4.16}) and integration by parts, we obtain
\begin{align*}
&\int\limits_{\mathbb{R}^N}R_0\Big[v_h-L_hR_0\Big]dy\\
=&-\big[V(x_h)+\lambda\big]^{-2}\int\limits_{\mathbb{R}^N}
\big(\frac{1}{p-1}v_h+\frac{1}{2}y\cdot\nabla v_h\big)
\big[V(hy+x_h)-V(x_h)+\frac{h}{2}y\cdot\nabla V(hy+x_h)\big]v_hdy\\
&+\big[V(x_h)+\lambda\big]^{-2}\int\limits_{\mathbb{R}^N}
\big(\frac{1}{p-1}v_h+\frac{1}{2}y\cdot\nabla v_h\big)
\frac{h}{2}y\cdot\nabla m(hy+x_h)v_h^pdy\\
=&\frac{1}{8}\big[V(x_h)+\lambda\big]^{-2}\int\limits_{\mathbb{R}^N}
\Big[3hy\cdot\nabla
V(hy+x_h)+h^2\sum\limits_{i,j=1}^N\partial_{ij}V(hy+x_h)y_iy_j\Big]v_h^2dy\\
&+\frac{N}{8(N+2)}\big[V(x_h)+\lambda\big]^{-2}\int\limits_{\mathbb{R}^N}
\Big[hy\cdot\nabla
m(hy+x_h)-h^2\sum\limits_{i,j=1}^N\partial_{ij}m(hy+x_h)y_iy_j\Big]v_h^{p+1}dy\,.
\end{align*}
Hence by (\ref{id4.10}), (\ref{id4.8}) and Taylor's formulas of
$V$ and $m$, we have
\begin{align}\label{id4.20}
&\int\limits_{\mathbb{R}^N}R_0\Big[v_h-L_hR_0\Big]dy\nonumber\\
=&\frac{1}{8}\big[V(x_h)+\lambda\big]^{-2}\int\limits_{\mathbb{R}^N}
\Big[4h^2\sum\limits_{i,j=1}^N\partial_{ij}V(x_0)y_iy_jw_{x_h}^2
+8h^4\sum\limits_{i,j=1}^N\partial_{ij}V(x_0)y_iy_jw_{x_h}\phi_2\nonumber\\
&+4h^4\sum\limits_{i,j,k=1}^N\partial_{ijk}V(x_0)x_{1,i}y_jy_kw_{x_h}^2
+h^4\sum\limits_{i,j,k,l=1}^N\partial_{ijkl}V(x_0)y_iy_jy_ky_lw_{x_h}^2\Big]dy\nonumber\\
&+\frac{N}{8(N+2)}\big[V(x_h)+\lambda\big]^{-2}\int\limits_{\mathbb{R}^N}
\Big[-\frac{h^4}{3}\sum\limits_{i,j,k,l=1}^N\partial_{ijkl}m(x_0)y_iy_jy_ky_l\Big]
w_{x_h}^{p+1}dy+\mathrm{o}(h^4).
\end{align}

For the sum
$\sum\limits_{j,k=1}^Nc_h^jc_h^k\int\limits_{\mathbb{R}^N}\partial_k
w_{x_h}\big(L_h\partial_j w_{x_h}\big)dy$, by (\ref{eq4.4}) and
(\ref{id4.17}), we obtain
\begin{eqnarray}\label{id4.21}
\sum\limits_{j,k=1}^Nc_h^jc_h^k\int\limits_{\mathbb{R}^N}\partial_k
w_{x_h}\big(L_h\partial_j
w_{x_h}\big)dy=-\frac{h^4}{N+2}\int\limits_{\mathbb{R}^N}w^{p+1}dy\nabla^2G(x_0)
\textbf{c}_0\cdot\textbf{c}_0+\mathrm{o}(h^4)\,.
\end{eqnarray}

For the integral $\int\limits_{\mathbb{R}^N}R_1\big(L_hR_1\big)dy$,
by (\ref{eq4.5}), $R_1\big(\frac{y}{\sqrt{V(x_h)+\lambda}}\big)$
satisfies
\begin{align}
&\Delta
R-R+pw^{p-1}R+\big[V(x_h)+\lambda\big]^{\frac{N}{4}-3}m(x_h)^{-\frac{N}{4}}
\sum\limits_{i,j=1}^N\partial_{ij}V(x_0)y_iy_jw\nonumber\\
&-\frac{1}{2}\big[V(x_h)+\lambda\big]^{\frac{N}{4}-2}m(x_h)^{-\frac{N}{4}-1}
\sum\limits_{i,j=1}^N\partial_{ij}m(x_0)y_iy_jw^p=0\,.
\end{align}
Hence
\begin{align}\label{id4.22}
&\int\limits_{\mathbb{R}^N}R_1\big(L_hR_1\big)dy=\int\limits_{\mathbb{R}^N}
R_1(L_{x_h}R_1)dy+\mathrm{O}(h^2)\nonumber\\
=&\big[V(x_h)+\lambda\big]^{-5}m(x_h)^{-\frac{N}{2}}\sum\limits_{i,j,k,l=1}^N
\partial_{ij}V(x_0)\partial_{kl}V(x_0)\int\limits_{\mathbb{R}^N}y_iy_jwL_0^{-1}(y_ky_lw)dy\nonumber\\
&-\big[V(x_h)+\lambda\big]^{-4}m(x_h)^{-\frac{N}{2}-1}\sum\limits_{i,j,k,l=1}^N
\partial_{ij}V(x_0)\partial_{kl}m(x_0)\int\limits_{\mathbb{R}^N}y_iy_jwL_0^{-1}(y_ky_lw^p)dy\\
&+\frac{1}{4}\big[V(x_h)+\lambda\big]^{-3}m(x_h)^{-\frac{N}{2}-2}\sum\limits_{i,j,k,l=1}^N
\partial_{ij}m(x_0)\partial_{kl}m(x_0)\int\limits_{\mathbb{R}^N}
y_iy_jw^pL_0^{-1}(y_ky_lw^p)dy+\mathrm{O}(h^2).\nonumber
\end{align}

As in Section 3, we have used the following identities:
\begin{align*}
\sum\limits_{i,j=1}^N\partial_{ij}V(x_0)\int\limits_{\mathbb{R}^N}y_iy_jw_{x_h}^2dy
=\frac{1}{N}\int\limits_{\mathbb{R}^N}|y|^2w_{x_h}^2dy\Delta
V(x_0)=0\,,
\end{align*}
\begin{align*}
&\sum\limits_{i,j=1}^N\partial_{ij}V(x_0)\int\limits_{\mathbb{R}^N}y_iy_jw_{x_h}\phi_2dy\\
=&\frac{1}{2}\big[V(x_h)+\lambda\big]^{-3}m(x_h)^{-\frac{N}{2}}\sum\limits_{i,j,k,l=1}^N
\partial_{ij}V(x_0)\partial_{kl}V(x_0)\int\limits_{\mathbb{R}^N}y_iy_jwL_0^{-1}(y_ky_lw)dy\\
&-\frac{1}{2}\big[V(x_h)+\lambda\big]^{-2}m(x_h)^{-\frac{N}{2}-1}\sum\limits_{i,j,k,l=1}^N
\partial_{ij}V(x_0)\partial_{kl}m(x_0)\int\limits_{\mathbb{R}^N}y_iy_jwL_0^{-1}(y_ky_lw^p)dy\,,
\end{align*}
\begin{align*}
\begin{cases}
&\sum\limits_{i,j,k=1}^N\partial_{ijk}V(x_0)x_{1,i}\int\limits_{\mathbb{R}^N}y_jy_kw_{x_h}^2dy
=\frac{1}{N}\int\limits_{\mathbb{R}^N}|y|^2w_{x_h}^2dy\nabla(\Delta
V)(x_0)\cdot \textbf{x}_1\,,\\
&\sum\limits_{i,j,k,l=1}^N\partial_{ijkl}V(x_0)\int\limits_{\mathbb{R}^N}y_iy_jy_ky_lw_{x_h}^2
=\frac{3}{N(N+2)}\int\limits_{\mathbb{R}^N}|y|^4w_{x_h}^2dy\Delta^2V(x_0)\,,\\
&\sum\limits_{i,j,k,l=1}^N\partial_{ijkl}m(x_0)\int\limits_{\mathbb{R}^N}y_iy_jy_ky_lw_{x_h}^{p+1}
=\frac{3}{N(N+2)}\int\limits_{\mathbb{R}^N}|y|^4w_{x_h}^{p+1}dy\Delta^2m(x_0)\,,
\end{cases}
\end{align*}
\begin{align*}
\begin{cases}
&\int\limits_{\mathbb{R}^N}
y^2_{N}wL_0^{-1}(y^2_Nw)dy=\frac{1}{N^2}\int\limits_{\mathbb{R}^N}r^2wL_0^{-1}(r^2w)dy
+\frac{2(N-1)}{N^2(N+2)}\int\limits_{\mathbb{R}^N}
r^2w\Phi_1(r)dy\,,\\
&\int\limits_{\mathbb{R}^N}
y^2_{N-1}wL_0^{-1}(y^2_Nw)dy=\frac{1}{N^2}\int\limits_{\mathbb{R}^N}r^2wL_0^{-1}(r^2w)dy
-\frac{2}{N^2(N+2)}\int\limits_{\mathbb{R}^N}r^2w\Phi_1(r)dy\,,\\
&\int\limits_{\mathbb{R}^N}
y_{N-1}y_NwL_0^{-1}(y_{N-1}y_Nw)dy=\frac{1}{N(N+2)}\int\limits_{\mathbb{R}^N}r^2w\Phi_1(r)dy\,,
\end{cases}
\end{align*}
\begin{align*}
\begin{cases}
&\int\limits_{\mathbb{R}^N}
y^2_{N}wL_0^{-1}(y^2_Nw^p)dy=\frac{1}{N^2}\int\limits_{\mathbb{R}^N}r^2wL_0^{-1}(r^2w^p)dy
+\frac{2(N-1)}{N^2(N+2)}\int\limits_{\mathbb{R}^N}
r^2w\Phi_0(r)dy\,,\\
&\int\limits_{\mathbb{R}^N}
y^2_{N-1}wL_0^{-1}(y^2_Nw^p)dy=\frac{1}{N^2}\int\limits_{\mathbb{R}^N}r^2wL_0^{-1}(r^2w^p)dy
-\frac{2}{N^2(N+2)}\int\limits_{\mathbb{R}^N}r^2w\Phi_0(r)dy\,,\\
&\int\limits_{\mathbb{R}^N}
y_{N-1}y_NwL_0^{-1}(y_{N-1}y_Nw^p)dy=\frac{1}{N(N+2)}\int\limits_{\mathbb{R}^N}r^2w\Phi_0(r)dy\,,
\end{cases}
\end{align*}
where $\Phi_0,\Phi_1$ satisfy
\begin{align*}
\begin{cases}
&\Phi_0''+\frac{N-1}{r}\Phi_0'-\Phi_0+pw^{p-1}\Phi_0-\frac{2N}{r^2}\Phi_0-r^2w^p=0,\,r\in(0,\infty)\,,\\
&\Phi_0(0)=\Phi_0'(0)=0\,,
\end{cases}
\end{align*}
and
\begin{align*}
\begin{cases}
&\Phi_1''+\frac{N-1}{r}\Phi_1'-\Phi_0+pw^{p-1}\Phi_1-\frac{2N}{r^2}\Phi_1-r^2w=0,\,r\in(0,\infty)\,,\\
&\Phi_1(0)=\Phi_1'(0)=0\,,
\end{cases}
\end{align*}
which can be proved as in Appendix B.

Therefore, combining (\ref{id4.9}), (\ref{id4.20}), (\ref{id4.21})
and (\ref{id4.22}), we obtain
\begin{align}\label{id4.23}
\frac{d\,''(\lambda)}{h^{N+4}}+\mathrm{o}(1)=H_2(x_0)+H_3(x_0)+H_4(x_0)\equiv
H(x_0)\,,
\end{align}
where
\begin{align}
H_2(x_0)=&\frac{3}{N(N+2)}\big[V(x_0)+\lambda\big]^{-5}m(x_0)^{-\frac{N}{2}}
\int\limits_{\mathbb{R}^N}|y|^2w\Phi_1(|y|)dy\|\nabla^2V(x_0)\|_2^2\nonumber\\
&-\frac{3}{N(N+2)}\big[V(x_0)+\lambda\big]^{-4}m(x_0)^{-\frac{N}{2}-1}
\int\limits_{\mathbb{R}^N}|y|^2w\Phi_0(|y|)dy\nabla^2V(x_0)\cdot\nabla^2m(x_0)\nonumber\\
&+\frac{1}{4N^2}\big[V(x_0)+\lambda\big]^{-3}m(x_0)^{-\frac{N}{2}-2}
\int\limits_{\mathbb{R}^N}|y|^2w^pL_0^{-1}(|y|^2w^p)dy|\Delta
m(x_0)|^2\nonumber\\
&+\frac{1}{2N(N+2)}\big[V(x_0)+\lambda\big]^{-3}m(x_0)^{-\frac{N}{2}-2}
\int\limits_{\mathbb{R}^N}|y|^2w^p\Phi_0(|y|)dy\|\nabla^2m(x_0)\|_2^2\nonumber\\
&-\frac{1}{2N^2(N+2)}\big[V(x_0)+\lambda\big]^{-3}m(x_0)^{-\frac{N}{2}-2}
\int\limits_{\mathbb{R}^N}|y|^2w^p\Phi_0(|y|)dy|\Delta m(x_0)|^2\,,
\end{align}
\begin{align}
H_3(x_0)=&\frac{1}{2N}\big[V(x_0)+\lambda\big]^{-3}m(x_0)^{-\frac{N}{2}}
\int\limits_{\mathbb{R}^N}|y|^2w^2dy\nabla(\Delta m)(x_0)\cdot
\textbf{x}_1\nonumber\\
&-\frac{1}{N+2}\int\limits_{\mathbb{R}^N}w^{p+1}dy
\textbf{c}_0\cdot\big[\nabla^2G(x_0)\big]^{-1}\textbf{c}_0\,,
\end{align}
\begin{align}
H_4(x_0)=&\frac{3}{8N(N+2)}\big[V(x_0)+\lambda\big]^{-4}m(x_0)^{-\frac{N}{2}}
\int\limits_{\mathbb{R}^N}|y|^4w^2dy\Delta^2V(x_0)\nonumber\\
&-\frac{1}{8(N+2)^2}\big[V(x_0)+\lambda\big]^{-3}m(x_0)^{-\frac{N}{2}-1}
\int\limits_{\mathbb{R}^N}|y|^4w^{p+1}dy\Delta^2m(x_0)\,.
\end{align}
Consequently, $p(d\,'')=1$ if $H(x_0)>0$, where $H(x_0)$ defined
in ~(\ref{id4.23}) involves the $i$-th derivatives (for $0\leq
i\leq 4$) of $V$ and $m$ at $x_0$. On the other hand, by
Theorem~\ref{thm4.1}, we have $n(L_h)=n+1$. Thus we complete the
proof of Theorem~{\ref{thm1.4}} by the orbital stability and
instability criteria of~\cite{[GS1]}-\cite{[GS2]}.

\section{Appendix A}
\ \ \ \ In this section, we want to prove (\ref{id2.13}) of
Section 2, i.e.
\begin{align}\label{id5.1}
\int_{\mathbb{R}^N}\left(L_h\partial_j w_{x_h}\right)\partial_k
w_{x_h}dy=\frac{h^2}{p+1}\int_{\mathbb{R}^N}w_{x_h}^{p+1}dy\partial
_{jk}m(x_0)+\mathrm{o}(h^2)\,.
\end{align}
\begin{proof}
Note that by Lemma~\ref{lem2.3} and \ref{lem2.4}, we obtain
\begin{align*}
L_h\partial_j
w_{x_h}=&\Big[m(hy+x_h)-m(x_h)\Big]pw_{x_h}^{p-1}\partial_j
w_{x_h}+m(hy+x_h)p(v_h^{p-1}-w_{x_h}^{p-1})\partial_j w_{x_h}\\
=&\frac{h^2}{2}\sum_{i,l=1}^N\partial_{il}m(x_0)y_iy_lpw_{x_h}^{p-1}\partial_jw_{x_h}
+h^2m(x_h)p(p-1)w_{x_h}^{p-2}\phi_2\partial_{j}w_{x_h}+\mathrm{o}(h^2)\,.
\end{align*}
Hence we may write the integral
$\int_{\mathbb{R}^N}\,\left(L_h\partial_j w_{x_h}\right)\partial_k
w_{x_h}dy$ as follows:
\begin{align}\label{id5.5}
\int_{\mathbb{R}^N}\left(L_h\partial_{j}w_{x_h}\right)\partial_{k}w_{x_h}\,dy=I_1+I_2+\mathrm{o}(h^2),
\end{align}
where
\begin{align}
I_1=&\frac{h^2}{2}\sum_{i,l=1}^N\partial_{il}m(x_0)\int_{\mathbb{R}^N}
y_iy_lpw_{x_h}^{p-1}\partial_{j}w_{x_h}\partial_{k}w_{x_h}dy\,,\label{id5.2}
\\
I_2=&h^2\int_{\mathbb{R}^N}m(x_h)p(p-1)w_{x_h}^{p-2}\phi_2\partial_{j}w_{x_h}\partial_{k}w_{x_h}dy\,.\label{id5.3}
\end{align}

Note that from~(\ref{eq2.2}), we have
\begin{equation}\label{id5.4}
\Big[\Delta-\lambda+m(x_h)pw_{x_h}^{p-1}\Big]\partial_{jk}w_{x_h}+
m(x_h)p(p-1)w_{x_h}^{p-2}\partial_{j}w_{x_h}\partial_{k}w_{x_h}=0\,.
\end{equation}
Hence by (\ref{eq2.6}), (\ref{id5.3}) and (\ref{id5.4}), we may use
integration by parts to get
\begin{align}\label{id5.6}
I_2=&-h^2\int_{\mathbb{R}^N}\phi_2\Big[\Delta-\lambda+m(x_h)pw_{x_h}^{p-1}\Big]
\partial_{jk}w_{x_h}dy\notag\\
=&-h^2\int_{\mathbb{R}^N}\partial_{jk}w_{x_h}\Big[\Delta-\lambda+m(x_h)pw_{x_h}^{p-1}\Big]\phi_2dy
\notag\\
=&\frac{h^2}{2}\sum_{i,l=1}^N\partial_{il}m(x_0)\int_{\mathbb{R}^N}
y_iy_l w_{x_h}^p
\partial_{jk}w_{x_h}dy\notag\\
=&-\frac{h^2}{2}
\sum_{i,l=1}^N\partial_{il}m(x_0)\int_{\mathbb{R}^N}
\frac{\partial(y_iy_lw_{x_h}^p)}{\partial y_j}\partial_{k}w_{x_h}dy\notag\\
=&-\frac{h^2}{2}\sum_{i,l=1}^N\partial_{il}m(x_0)\int_{\mathbb{R}^N}
y_iy_lpw_{x_h}^{p-1}\partial_{j}w_{x_h}\partial_{k}w_{x_h}dy
-h^2\partial_{jk}m(x_0)\int_{\mathbb{R}^N}y_kw_{x_h}^{p}\partial_{k}w_{x_h}dy\notag\\
=&-\frac{h^2}{2}\sum_{i,l=1}^N\partial_{il}m(x_0)\int_{\mathbb{R}^N}
y_iy_lpw_{x_h}^{p-1}\partial_{j}w_{x_h}\partial_{k}w_{x_h}dy
+\frac{h^2}{p+1}\partial_{jk}m(x_0)\int_{\mathbb{R}^N}w_{x_h}^{p+1}dy\,.
\end{align}
Combining (\ref{id5.5}), (\ref{id5.2}) and (\ref{id5.6}), we
obtain~(\ref{id5.1}).
\end{proof}

\section{Appendix B}
\ \ \ \ In this section, we prove (\ref{id3.21-1}),
(\ref{id3.21-2}) and (\ref{id3.21-3}) of Section~3, i.e.
\begin{align}
&\int\limits_{\mathbb{R}^N}
y^2_{N}w^pL_0^{-1}(y^2_Nw^p)dy=\frac{1}{N^2}\int\limits_{\mathbb{R}^N}r^2w^pL_0^{-1}(r^2w^p)dy
+\frac{2(N-1)}{N^2(N+2)}\int\limits_{\mathbb{R}^N}
r^2w^p\Phi_0(r)dy\,,\label{id6.1}\\
&\int\limits_{\mathbb{R}^N}
y^2_{N-1}w^pL_0^{-1}(y^2_Nw^p)dy=\frac{1}{N^2}\int\limits_{\mathbb{R}^N}r^2w^pL_0^{-1}(r^2w^p)dy
-\frac{2}{N^2(N+2)}\int\limits_{\mathbb{R}^N}r^2w^p\Phi_0(r)dy\,,\label{id6.2}\\
&\int\limits_{\mathbb{R}^N}
y_{N-1}y_Nw^pL_0^{-1}(y_{N-1}y_Nw^p)dy=\frac{1}{N(N+2)}
\int\limits_{\mathbb{R}^N}r^2w^p\Phi_0(r)dy\,,\label{id6.3}
\end{align}
where $r:=|y|$ and $\Phi_0$ satisfies
\begin{equation}\label{eq6.1}
\left\{ \begin{aligned}
&\Phi_0''+\frac{N-1}{r}\Phi_0'-\Phi_0+pw^{p-1}\Phi_0-\frac{2N}{r^2}\Phi_0-r^2w^p=0,\,r\in(0,\infty),\\
&\Phi_0(0)=\Phi_0'(0)=0.
\end{aligned} \right.
\end{equation}

\begin{proof}
From~(\ref{eq6.1}), it is easy to check that
\begin{equation}
L_0\left[\Phi_0\frac{y_N^2}{r^2}+\frac{1}{N}L_0^{-1}(r^2w^p)
-\frac{1}{N}\Phi_0\right]=y_N^2w^p\,,{\rm{and}}\,
L_0\left[\Phi_0\frac{y_{N-1}y_N}{r^2}\right]=y_{N-1}y_Nw^p.
\end{equation}
Then using the polar coordinate, we obtain
\begin{align*}
&\int\limits_{\mathbb{R}^N} y^2_{N}w^pL_0^{-1}(y^2_Nw^p)dy\\
=&\int\limits_{\mathbb{R}^N}
y^2_{N}w_{x_0}^p\left[\Phi_0(r)\frac{y_N^2}{r^2}-\frac{1}{N}\Phi_0(r)+\frac{1}{N}L_0^{-1}(r^2w^p)\right]dy\nonumber\\
=&\int\limits_{\mathbb{R}^N}
r^2\cos^2\theta_{N-1}w^p\left[\Phi_0(r)\frac{r^2\cos^2\theta_{N-1}}{r^2}-\frac{1}{N}\Phi_0(r)+\frac{1}{N}L_0^{-1}(r^2w^p)\right]dy\nonumber\\
=&\frac{\int\limits_{0}^{\pi}\cos^4\theta_{N-1}\sin^{N-2}\theta_{N-1}d\theta_{N-1}}
{\int\limits_{0}^{\pi}\sin^{N-2}\theta_{N-1}d\theta_{N-1}}\int\limits_{\mathbb{R}^N}r^2w^p\Phi_0(r)dy\nonumber\\
&+\frac{\int\limits_{0}^{\pi}\cos^2\theta_{N-1}\sin^{N-2}\theta_{N-1}d\theta_{N-1}}
{\int\limits_{0}^{\pi}\sin^{N-2}\theta_{N-1}d\theta_{N-1}}
\int\limits_{\mathbb{R}^N}r^2w_{x_0}^p\left[-\frac{1}{N}\Phi_0(r)+\frac{1}{N}L_0^{-1}(r^2w_0^p)\right]dy\nonumber\\
=&\frac{3}{N(N+2)}\,\int\limits_{\mathbb{R}^N}r^2w^p\Phi_0(r)dy
+\frac{1}{N}\,\int\limits_{\mathbb{R}^N}r^2w^p\left[-\frac{1}{N}\Phi_0(r)+\frac{1}{N}L_0^{-1}(r^2w^p)\right]dy\nonumber\\
=&\frac{1}{N^2}\int\limits_{\mathbb{R}^N}r^2w^pL_0^{-1}(r^2w^p)dy
+\frac{2(N-1)}{N^2(N+2)}\int\limits_{\mathbb{R}^N}r^2w^p\Phi_0(r)dy\,\nonumber.
 \end{align*}
This completes the proof of (\ref{id6.1}). Similarly, one may obtain
(\ref{id6.2}) and (\ref{id6.3}), respectively.
\end{proof}

\section{Appendix C}
\ \ \ \ In this section, we prove (\ref{eq4.4}) of Section~4, i.e.
\begin{equation}\label{id7.11}
\int\limits_{\mathbb{R}^N}\partial_{k} w_{x_h}L_h\big(\partial_{j}
w_{x_h}+h\psi_j\big)dy=-\frac{h^2}{N+2}\int\limits_{\mathbb{R}^N}w^{p+1}dy
\partial_{jk}G(x_0)+\mathrm{o}(h^2).
\end{equation}
\begin{proof}
Note that by (\ref{id4.2}), (\ref{id4.3}) and (\ref{eq4.3}), we
obtain
\begin{align*}
L_h\partial_j w_{x_h}=&L_{x_h}\partial_j
w_{x_h}+\Big[m(hy+x_h)-m(x_h)\Big]pw_{x_h}^{p-1}\partial_j
w_{x_h}\\
&+m(hy+x_h)p(v_h^{p-1}-w_{x_h}^{p-1})\partial_j w_{x_h}
-\Big[V(hy+x_h)-V(x_h)\Big]\partial_jw_{x_h}\\
=&h\Big[y\cdot\nabla m(x_h)pw_{x_h}^{p-1}
+m(x_h)p(p-1)w_{x_h}^{p-2}\phi_1-y\cdot\nabla V(x_h)\Big]\partial_jw_{x_h}\\
&+h^2\Big[\frac{1}{2}\sum_{i,l}^N\partial_{il}m(x_h)y_iy_lpw_{x_h}^{p-1}
+y\cdot\nabla
m(x_h)p(p-1)w_{x_h}^{p-2}\phi_1
+m(x_h)p(p-1)w_{x_h}^{p-2}\phi_2\\
&+\frac{1}{2}m(x_h)p(p-1)(p-2)w_{x_h}^{p-3}\phi_1^2
-\frac{1}{2}\sum_{i,l}^N\partial_{il}V(x_h)y_iy_l\Big]\partial_jw_{x_h}+\mathrm{o}(h^2)\,,
\end{align*}
and
\begin{align*}
L_h\psi_j=&L_{x_h}\psi_j+\Big[m(hy+x_h)-m(x_h)\Big]pw_{x_h}^{p-1}\psi_j\\
&+m(hy+x_h)p(v_h^{p-1}-w_{x_h}^{p-1})\psi_j
-\Big[V(hy+x_h)-V(x_h)\Big]\psi_j\\
=&-\Big[y\cdot\nabla m(x_h)pw_{x_h}^{p-1}
+m(x_h)p(p-1)w_{x_h}^{p-2}\phi_1-y\cdot\nabla V(x_h)\Big]\partial_jw_{x_h}\\
&+h\Big[y\cdot\nabla m(x_h)pw_{x_h}^{p-1}
+m(x_h)p(p-1)w_{x_h}^{p-2}\phi_1-y\cdot\nabla
V(x_h)\Big]\psi_j+\mathrm{O}(h^2).
\end{align*}
Hence we may write the integral
$\int\limits_{\mathbb{R}^N}\partial_{k} w_{x_h}L_h\big(\partial_{j}
w_{x_h}+h\psi_j\big)dy$ as follows:
\begin{align}\label{id7.1}
\int\limits_{\mathbb{R}^N}\partial_{k} w_{x_h}L_h\big(\partial_{j}
w_{x_h}+h\psi_j\big)dy=I_0+I_1+I_2+\mathrm{o}(h^2),
\end{align}
where
\begin{align}
I_0=&h^2\int_{\mathbb{R}^N}\Big[y\cdot\nabla m(x_h)pw_{x_h}^{p-1}
+m(x_h)p(p-1)w_{x_h}^{p-2}\phi_1-y\cdot\nabla
V(x_h)\Big]\psi_j\partial_k w_{x_h}dy\,,\label{id7.2}\\
I_1=&h^2\int_{\mathbb{R}^N}
\Big[\frac{1}{2}\sum_{i,l}^N\partial_{il}m(x_h)y_iy_lpw_{x_h}^{p-1}
+y\cdot\nabla m(x_h)p(p-1)w_{x_h}^{p-2}\phi_1\nonumber\\
&\quad\quad\quad+\frac{1}{2}m(x_h)p(p-1)(p-2)w_{x_h}^{p-3}\phi_1^2
-\frac{1}{2}\sum_{i,l}^N\partial_{il}V(x_h)y_iy_l\Big]\partial_j
w_{x_h}\partial_k w_{x_h}dy\,,\label{id7.3}\\
I_2=&h^2\int_{\mathbb{R}^N}m(x_h)p(p-1)w_{x_h}^{p-2}\phi_2\partial_j
w_{x_h}\partial_k w_{x_h}dy\,.\label{id7.4}
\end{align}

Note that from~(\ref{eq2.2}), we have
\begin{equation}\label{id7.5}
\Big[\Delta-\big(V(x_h)+\lambda\big)+m(x_h)pw_{x_h}^{p-1}\Big]\partial_{jk}w_{x_h}+
m(x_h)p(p-1)w_{x_h}^{p-2}\partial_{j}w_{x_h}\partial_{k}w_{x_h}=0\,.
\end{equation}
Hence by (\ref{eq4.7}), (\ref{id7.3}) and (\ref{id7.4}), we may use
integration by parts to get
\begin{align}\label{id7.6}
I_2=&-h^2\int_{\mathbb{R}^N}\phi_2\Big[\Delta-\big(V(x_h)+\lambda\big)+m(x_h)pw_{x_h}^{p-1}\Big]
\partial_{jk}w_{x_h}dy\notag\\
=&-h^2\int_{\mathbb{R}^N}\partial_{jk}w_{x_h}
\Big[\Delta-\big(V(x_h)+\lambda\big)+m(x_h)pw_{x_h}^{p-1}\Big]\phi_2dy
\notag\\
=&h^2\int_{\mathbb{R}^N}\Big[-y\cdot\nabla
V(x_h)\phi_1-\frac{1}{2}\sum\limits_{i,l=1}^N\partial_{il}V(x_h)y_iy_lw_{x_h}
+y\cdot\nabla m(x_h)pw_{x_h}^{p-1}\phi_1\nonumber\\
&\quad\quad\quad+\frac{1}{2}\sum\limits_{i,l=1}^N\partial_{il}m(x_h)y_iy_lw_{x_h}^p
+\frac{1}{2}m(x_h)p(p-1)w_{x_h}^{p-2}\phi_1^2\Big]
\partial_{jk}w_{x_h}dy\notag\\
=&h^2\int_{\mathbb{R}^N}\Big[\partial_jV(x_h)\phi_1+y\cdot\nabla
V(x_h)\partial_j\phi_1
+\frac{1}{2}\sum\limits_{i,l=1}^N\partial_{il}V(x_h)y_iy_l\partial_jw_{x_h}
+\partial_{jk}V(x_h)y_kw_{x_h}\nonumber\\
&\quad\quad\quad-\partial_jm(x_h)pw_{x_h}^{p-1}\phi_1-y\cdot\nabla m(x_h)p(p-1)w_{x_h}^{p-2}\phi_1\partial_jw_{x_h}-y\cdot\nabla m(x_h)pw_{x_h}^{p-1}\partial_j\phi_1\nonumber\\
&\quad\quad\quad-\frac{1}{2}\sum\limits_{i,l=1}^N\partial_{il}m(x_h)y_iy_lpw_{x_h}^{p-1}\partial_jw_{x_h}
-\partial_{jk}m(x_h)y_kw_{x_h}^p\nonumber\\
&\quad\quad\quad-\frac{1}{2}m(x_h)p(p-1)(p-2)w_{x_h}^{p-3}\phi_1^2\partial_jw_{x_h}
-m(x_h)p(p-1)w_{x_h}^{p-2}\phi_1\partial_j\phi_1\Big]
\partial_kw_{x_h}dy\notag\\
=&-I_1-h^2\int_{\mathbb{R}^N}\Big[y\cdot\nabla m(x_h)pw_{x_h}^{p-1}
+m(x_h)p(p-1)w_{x_h}^{p-2}\phi_1-y\cdot\nabla
V(x_h)\Big]\partial_j\phi_1\partial_k w_{x_h}dy\nonumber\\
&+h^2\int_{\mathbb{R}^N}\Big[\partial_jV(x_h)\phi_1
+\partial_{jk}V(x_h)y_kw_{x_h} -\partial_jm(x_h)pw_{x_h}^{p-1}\phi_1
-\partial_{jk}m(x_h)y_kw_{x_h}^p\Big]
\partial_kw_{x_h}dy.
\end{align}
Note that from~(\ref{eq4.6}), we have
\begin{align}\label{id7.7}
&\Delta(\partial_j\phi_1)-\left[V(x_0)+\lambda\right]\partial_j\phi_1+m(x_0)pw_{x_0}^{p-1}\partial_j\phi_1
+m(x_0)p(p-1)w_{x_0}^{p-2}\phi_1\partial_jw_{x_0}\nonumber\\
&-y\cdot\nabla
V(x_0)\partial_jw_{x_0}-\partial_jV(x_0)w_{x_0}+y\cdot\nabla
m(x_0)pw_{x_0}^{p-1}\partial_jw_{x_0}+\partial_jm(x_0)w_{x_0}^p=0\,,
\end{align}
and by direct computation,
\begin{align}\label{id7.8}
\begin{cases}
&L_{x_0}w_{x_0}=(p-1)m(x_0)w_{x_0}^p\,,\\
&L_{x_0}(\frac{1}{p-1}w_{x_0}+\frac{1}{2}y\cdot\nabla
w_{x_0})=\left[V(x_0)+\lambda\right]w_{x_0}\,.
\end{cases}
\end{align}

Thus we may use~(\ref{id7.2})-(\ref{id7.8}) and integration by parts
to get
\begin{align}\label{id7.9}
&I_0+I_1+I_2\nonumber\\
=&h^2\int_{\mathbb{R}^N}\Big[y\cdot\nabla m(x_h)pw_{x_h}^{p-1}
+m(x_h)p(p-1)w_{x_h}^{p-2}\phi_1-y\cdot\nabla
V(x_h)\Big]\Big(\psi_j-\partial_j\phi_1\Big)\partial_k
w_{x_h}dy\nonumber\\
&+h^2\int_{\mathbb{R}^N}\Big[\partial_jV(x_h)\phi_1
+\partial_{jk}V(x_h)y_kw_{x_h} -\partial_jm(x_h)pw_{x_h}^{p-1}\phi_1
-\partial_{jk}m(x_h)y_kw_{x_h}^p\Big]
\partial_kw_{x_h}dy\nonumber\\
=&-h^2\int_{\mathbb{R}^N}\Big(\psi_j-\partial_j\phi_1\Big)L_{x_h}\psi_kdy
+h^2\int_{\mathbb{R}^N}\Big[\partial_jV(x_h)
-\partial_jm(x_h)pw_{x_h}^{p-1}\Big]\phi_1
\partial_kw_{x_h}dy\nonumber\\
&+h^2\int_{\mathbb{R}^N}\Big[\partial_{jk}V(x_h)y_kw_{x_h}
-\partial_{jk}m(x_h)y_kw_{x_h}^p\Big]
\partial_kw_{x_h}dy\nonumber\\
=&h^2\int_{\mathbb{R}^N}\Big[\partial_jV(x_0)w_{x_0}
-\partial_jm(x_0)w_{x_0}^p\Big]\psi_kdy
-h^2\int_{\mathbb{R}^N}\Big[\partial_jV(x_h)w_{x_h}
-\partial_jm(x_h)w_{x_h}^p\Big]\partial_k\phi_1dy\nonumber\\
&-h^2\int_{\mathbb{R}^N}\Big[\frac{1}{2}\partial_{jk}V(x_h)w_{x_h}^2
-\frac{1}{p+1}\partial_{jk}m(x_h)w_{x_h}^{p+1}\Big]dy+\mathrm{o}(h^2)\nonumber\\
=&h^2\int_{\mathbb{R}^N}\Big[\partial_jV(x_0)\big(V(x_0)+\lambda\big)^{-1}\big(\frac{1}{p-1}w_{x_0}+\frac{1}{2}y\cdot\nabla
w_{x_0}\big)\nonumber\\
&\quad\quad\quad-\frac{1}{p-1}m(x_0)^{-1}\partial_jm(x_0)w_{x_0}\Big]L_{x_0}\Big(\psi_k-\partial_k\phi_1\Big)dy\nonumber\\
&-h^2\int_{\mathbb{R}^N}\Big[\frac{1}{2}\partial_{jk}V(x_0)w_{x_0}^2
-\frac{1}{p+1}\partial_{jk}m(x_0)w_{x_0}^{p+1}\Big]dy+\mathrm{o}(h^2)\nonumber\\
=&-h^2\int_{\mathbb{R}^N}\Big[\partial_jV(x_0)\big(V(x_0)+\lambda\big)^{-1}\big(\frac{1}{p-1}w_{x_0}+\frac{1}{2}y\cdot\nabla
w_{x_0}\big)\nonumber\\
&\quad\quad\quad-\frac{1}{p-1}m(x_0)^{-1}\partial_jm(x_0)w_{x_0}\Big]\Big[\partial_kV(x_0)w_{x_0}
-\partial_km(x_0)w_{x_0}^p\Big]dy\nonumber\\
&-h^2\int_{\mathbb{R}^N}\Big[\frac{1}{2}\partial_{jk}V(x_0)w_{x_0}^2
-\frac{1}{p+1}\partial_{jk}m(x_0)w_{x_0}^{p+1}\Big]dy+\mathrm{o}(h^2)\nonumber\\
=&-h^2\Big[\big(\frac{1}{p-1}-\frac{N}{4}\big)\left[V(x_0)+\lambda\right]^{-1}
\partial_jV(x_0)\partial_kV(x_0)\nonumber\\
&\quad\quad\quad-\frac{1}{p-1}m(x_0)^{-1}\partial_jm(x_0)\partial_kV(x_0)\Big]\int_{\mathbb{R}^N}w_{x_0}^2dy\nonumber\\
&+h^2\Big[\big(\frac{1}{p-1}-\frac{1}{2}\frac{N}{p+1}\big)\left[V(x_0)+\lambda\right]^{-1}
\partial_jV(x_0)\partial_km(x_0)\nonumber\\
&\quad\quad\quad-\frac{1}{p-1}m(x_0)^{-1}\partial_jm(x_0)\partial_km(x_0)\Big]\int_{\mathbb{R}^N}w_{x_0}^{p+1}dy\nonumber\\
&-h^2\int_{\mathbb{R}^N}\Big[\frac{1}{2}\partial_{jk}V(x_0)w_{x_0}^2
-\frac{1}{p+1}\partial_{jk}m(x_0)w_{x_0}^{p+1}\Big]dy+\mathrm{o}(h^2)\,.
\end{align}
Recall that
\begin{align*}
\begin{cases}
w_{x_0}(y)=\big[V(x_0)+\lambda\big]^{\frac{1}{p-1}}m(x_0)^{-\frac{1}{p-1}}w(\sqrt{V(x_0)+\lambda}y)\,,\\
m(x_0)\nabla V(x_0)=\frac{N}{2}\left[V(x_0)+\lambda\right]\nabla
m(x_0)\,,\\
\partial_{ij}G(x_0)=m(x_0)^{-\frac{N}{2}-1}\Big[m(x_0)\partial_{ij}V(x_0)
+(1-\frac{N}{2})\partial_iV(x_0)\partial_jm(x_0)-\frac{N}{2}\left[V(x_0)
+\lambda\right]\partial_{ij}m(x_0)\Big]\,,
\end{cases}
\end{align*}
and the integral identity
\begin{equation*}
\left[V(x_0)+\lambda\right]\int_{\mathbb{R}^N}w_{x_0}^2dy=\frac{2}{N+2}m(x_0)\int_{\mathbb{R}^N}w_{x_0}^{p+1}dy\,.
\end{equation*}
Combining (\ref{id7.1}) and (\ref{id7.9}), we obtain~(\ref{id7.11}).
\end{proof}

\end{document}